\documentclass[final,leqno]{siamltex}
\usepackage{amsmath}
\usepackage{amsfonts}
\usepackage{graphicx}
\usepackage{nicefrac}
\newtheorem{remark}{Remark}[section]
\def\Jc{{\mathcal J}}

\def\Tc{{\mathcal T}}

\def\Asf{{\mathsf A}}
\def\Bsf{{\mathsf B}}
\def\Csf{{\mathsf C}}

\def\Msf{{\mathsf M}}
\def\Nsf{{\mathsf N}}

\def\Qsf{{\mathsf Q}}
\def\Isf{{\mathsf I}}
\def\Psf{{\mathsf P}}
\def\Rsf{{\mathsf R}}

\def\div{{\mbox{div }}}
\def\vep{\varepsilon}

\def\n{\nabla}

\def\p{{\partial}}

\def\op{\overline{p}}

\def\nab{\nabla}
\def\Ome{\Omega}

\def\Del{\Delta}
	\def\vphi{\varphi}

\def\vphihm{\vphi^{m}_h}
\def\phm{p^{m}_h}

\def\of{\overline{f}}

\newcommand{\norm}[2]{\left\Vert#1\right\Vert_{#2}}
\newcommand{\prodt}[2]{\left(#1,#2\right)}
\newcommand{\dual}[2]{\left\langle#1, #2\right\rangle}

\newcommand{\bfphi}{\mbox{\boldmath$\phi$}}
\newcommand{\bfpsi}{\mbox{\boldmath$\psi$}}
\newcommand{\bfmu}{\mbox{\boldmath$\mu$}}
\newcommand{\bfnu}{\mbox{\boldmath$\nu$}}

\newcommand{\bfxi}{\mbox{\boldmath$\xi$}}
\newcommand{\bfvphi}{\mbox{\boldmath$\varphi$}}

\def\bu{\mathbf{u}}
\def\bx{\mathbf{x}}
\def\bp{\mathbf{p}}
\def\bq{\mathbf{q}}
\def\obu{\overline{\bu}}
\def\omu{\overline{\mu}}
\def\ovphi{\overline{\vphi}}
\def\oovphi{\overline{\overline{\vphi}}}

\def\bg{\mathbf{g}}
\def\bn{\mathbf{n}}
\def\bB{\mathbf{B}}

\def\bL{\mathbf{L}}

\def\buhm{\bu^m_h}

\begin{document}

\title{Analysis of a Darcy-Cahn-Hilliard Diffuse Interface Model for 
the Hele-Shaw Flow and its Fully Discrete Finite Element Approximation}
\markboth{XIAOBING FENG AND STEVEN WISE}{FINITE ELEMENT METHODS FOR 
HELE-SHAW FLOW}

\author{
Xiaobing Feng\thanks{Department of Mathematics, 
The University of Tennessee, Knoxville, TN 37996, U.S.A.
(xfeng@math.utk.edu). The work of this author was partially supported by 
the NSF grant DMS-0710831.}
\and
Steven Wise\thanks{Department of Mathematics, 
The University of Tennessee, Knoxville, TN 37996, U.S.A.
(swise@math.utk.edu). The work of this author was partially supported by the NSF grant DMS-0818030.}
}

\maketitle

\begin{abstract}
In this paper we present PDE and finite element analyses for  
a system of partial differential equations (PDEs) consisting of the 
Darcy equation and
the Cahn-Hilliard equation, which arises as a diffuse interface model
for the two phase Hele-Shaw flow. In the model the two sets of equations 
are coupled through an extra phase induced force term in the 
Darcy equations and a fluid induced transport term in the Cahn-Hilliard 
equation. We propose a fully discrete implicit finite element method for 
approximating the PDE system, which consists of the implicit Euler
method combined with a convex splitting energy strategy for the temporal
discretization, the standard finite element discretization for the pressure 
and a split (or mixed) finite element discretization for the fourth order 
Cahn-Hilliard equation. 
It is shown that the proposed numerical method satisfies 
a mass conservation law in addition to a discrete energy law that 
mimics the basic energy law for the Darcy-Cahn-Hilliard phase 
field model and holds uniformly in the phase field parameter $\varepsilon$. 
With help of the discrete energy law, we first prove that the fully discrete 
finite method is unconditionally energy stable and uniquely solvable at 
each time step. We then show that, using the compactness method,
the finite element solution has an accumulation point that is 
a weak solution of the PDE system. As a result, 
the convergence result also provides a constructive proof of the existence 
of global-in-time weak solutions to the Darcy-Cahn-Hilliard phase field model
in both two and three dimensions. 
%
%
Finally, we propose a nonlinear multigrid iterative algorithm 
to solve the finite element equations at each time step.
Numerical experiments based on the overall solution method of combining
the proposed finite element discretization and a nonlinear 
multigrid solver are presented to validate the theoretical
results and to show the effectiveness of the proposed fully discrete 
finite element method for approximating the Darcy-Cahn-Hilliard phase 
field model.
\end{abstract}

\begin{keywords}
Two phase Hele-Shaw flow, diffuse interface model, Darcy law, Cahn-Hilliard 
equation, energy splitting, finite element method, nonlinear multigrid.
\end{keywords}
                                                                                
\begin{AMS}
65M60, 35K55, 76D05
\end{AMS}
                                                                                
\section{Introduction}\label{sec-1}
Hele-Shaw flow refers to the motion of (one or more) viscous fluids
between two flat parallel plates separated by an infinitesimally small gap. 
Such a physical setup is often called a Hele-Shaw cell and was originally
designed by Hele-Shaw to study two dimensional potential flows \cite{hele-shaw}. 
Various fluid mechanics problems can be approximated by 
Hele-Shaw flows and thus the research of those flows is of
great practical importance. In addition, the relative simplicity 
of the governing equations of these flows makes Hele-Shaw flows
ideal test cases in which rigorous mathematical theory and 
efficient numerical methods can be developed for studying 
interfacial dynamics --- such as the formation of singularities and topological 
changes --- in immiscible fluids (cf. \cite{llg02a,llg02b,Nie_Tian98} 
and the references therein).

The governing equation of Hele-Shaw flow is identical to that of the inviscid potential flow and to the flow of fluids through a porous medium, because the gap-averaged velocity of the fluid is given by Darcy's law.  Specifically, the two phase Hele-Shaw flow takes the form (cf. \cite{llg02a} and the references therein):
	\begin{alignat}{4}
\bu &=-\frac{1}{12 \eta} \bigl(\nab p-\rho \bg \bigr) ,   &\quad  \div \bu &=0 &&\qquad \mbox{in } \Ome_T\setminus \Gamma_t,
	\label{eq1.1} 
	\\
[p] &=\gamma \kappa , &\quad [\bu\cdot \bn] &=0 &&\qquad \mbox{on } \Gamma_t,
	\label{eq1.3}
	\end{alignat}
with a given set of initial and boundary conditions. Here $\Ome_T=\Ome\times (0,T)$, where $\Ome\subset\mathbb{R}^2$ is a bounded domain.   $\Gamma_t$ denotes the interface  between the fluids at the time $t$ with the normal $\bn$. $\bu$ is the fluid velocity and $p$ stands for the pressure of the fluids.  The symbol $[p]$ stands for the jump of $p$ across the interface $\Gamma_t$.   $\eta$ is the viscosity, which may have different (positive constant) values on both sides of $\Gamma_t$.  ${\bf g}$ is the gravitational force per unit mass; and $\rho$ is the mass density of the fluid, which again can take different (positive constant) values on both sides of the interface.  Equation~\eqref{eq1.1}a is Darcy's law \cite{bear}, and \eqref{eq1.1}b implies that the fluids  are incompressible. Equations~\eqref{eq1.3}a and \eqref{eq1.3}b are the boundary conditions at the fluid-fluid interface, which represent the mathematical descriptions of the balance of the surface tension forces and the balance of mass, respectively. Equation~\eqref{eq1.3}a is called the Laplace-Young condition, where $\gamma$ is the dimensionless surface tension coefficient and $\kappa$ is the (mean) curvature of the  interface $\Gamma_t$.  Notice that the tangential component of the velocity $\bu$ may experience a jump across the interface~\cite{llg02a}.

Computationally, the above moving interface problem is difficult
to solve directly due to the existence of the surface tension on 
the interface. In addition, during the evolution the fluid interface 
may experience topological changes such as self-intersection, pinch-off,
splitting, and fattening. When that happens, the classical solution of
the moving interface problem ceases to exist. In such cases
it is very delicate matter to develop a proper notion of generalized
solutions, and it becomes even more challenging to compute the generalized 
solutions when they can be defined. 

To overcome the difficulties, an alternative approach for solving
moving interface problems is the {\em diffuse interface 
theory}, which was originally developed as methodology for
modeling and approximating solid-liquid phase transitions in
which the effects of surface tension and non-equilibrium
thermodynamic behavior may be important at the interface.
In the theory, the interface is represented as a thin layer
of finite thickness, as opposed to a sharp interface. Such an idea 
dates to Poisson, Gibbs, Rayleigh, van der Waals,
and Korteweg (see \cite{mcfadden,llg02a} and the references therein).
The approach then uses an auxiliary function (called the phase field function) to indicate the ``phase". The phase field function, denoted by 
$\vphi$ below, assumes distinct
values in the bulk phases away from the interfacial region, through
which the phase function varies smoothly.  The interface itself
can be associated with an intermediate contour/level set of the
phase function (cf. \cite{am2,alikakos94,cahn58,lt,mcfadden} 
and the references therein). Generally speaking, the diffuse interface models are expected to converge to some corresponding sharp interface 
models as the width of the interfacial layer tends to zero.

The diffuse interface model for Hele-Shaw flows to be studied 
in this paper is given as follows:
\begin{alignat}{2}\label{eq1.5}
\bu &=-\nab p -\gamma \varphi \nab \mu &&\quad \mbox{in } \Ome_T,\\
\div \bu &=0 &&\quad \mbox{in } \Ome_T, \label{eq1.6}\\
\varphi_t +\bu\cdot \nab \varphi - \vep\Del \mu &=0  
&&\quad \mbox{in } \Ome_T, \label{eq1.7} \\
\mu &= -\vep \Del \varphi + \frac{1}{\vep} f(\varphi) 
&&\quad \mbox{in } \Ome_T, \label{eq1.8}
\end{alignat}
where $f(\varphi)=F'(\varphi)$ and $F(\vphi)=\frac14(\varphi^2-1)^2$ is
the so-called double-well (potential) energy density, and $0< \varepsilon <1$ is a fixed constant.
To close the system, we impose the following initial and boundary conditions
\begin{alignat}{2}
\frac{\p p}{\p \bn} = \frac{\p \mu}{\p \bn} 
= \frac{\p \varphi}{\p \bn} &=0 &&\qquad\mbox{on}\quad \p\Ome_T:=\p\Omega\times (0,T],
\label{eq1.9} \\
\vphi(\, \cdot \, ,0) &=\vphi_0^\vep(\, \cdot\, ) &&\qquad\mbox{in}\quad \Ome.
\label{eq1.11}
\end{alignat}
Note that we have suppressed the superscript $\vep$ in
$(\bu^\vep,p^\vep,\vphi^\vep)$ for the sake of notational simplicity. 
Although $\Ome$ is a two-dimensional domain in the original Hele-Shaw problem,
in this paper we consider $\Ome\subset \mathbb{R}^d$ $(d=2,3)$ 
because the three-dimensional problem also has a mathematical interest
and arises from biological applications~\cite{wise10}.
Here the vector $\bu(x,t)\in \mathbb{R}^d$ and the scalar 
$p(x,t) \in \mathbb{R}$ denote the velocity and the pressure of 
the fluid mixture at the space-time point $(x,t)$, respectively.
The variables $\vphi(x,t),\, \mu(x,t) \in \mathbb{R}$ are the phase field 
function and the chemical potential, respectively. $\vphi$ assumes distinct
values --- namely, $\pm 1$ based on our choice of $F(\vphi)$ --- in the 
bulk phases away from a thin layer of width $O(\vep)$.  This thin layer is 
called the diffuse interface region.  It is natural to define the zero
level curve of $\vphi$, $\Gamma_\vep(t) = \left\{x\in\mathbb{R}^d\, \middle| \, \vphi(x,t) = 0\right\}$, as  the  $d-1$ dimensional interface.
 Eq.~\eqref{eq1.5} with $\gamma=0$ is the Darcy equation \cite{bear}.  
\eqref{eq1.7} and \eqref{eq1.8} without the convection term 
$\bu\cdot\nab \vphi$ is the Cahn-Hilliard 
equation \cite{cahn58,elliott86,mcfadden}.  Note that if $\gamma=0$, the velocity vanishes, and the Cahn-Hilliard equation results.

The system \eqref{eq1.5}--\eqref{eq1.11} is a special case
of the BHSCH (Boussinesq-Hele-Shaw-Cahn-Hilliard) model
proposed by Lee, Lowengrub, and Goodman in \cite{llg02a}.  They showed, using formal asymptotics, that solutions of the BHSCH system converge to those of the Hele-Shaw model \eqref{eq1.1}--\eqref{eq1.3} as their interfacial parameter converges to zero.   We note that the pressure $p$ in \eqref{eq1.5} has a different scaling from that in the BHSCH model in~\cite{llg02a}.  To obtain a similarly scaled pressure, one can simply introduce a redefined pressure in our model via $\tilde{p} = p +\gamma\vphi\mu$.  We shall refer \eqref{eq1.5}--\eqref{eq1.11} as  the 
DCH (Darcy-Cahn-Hilliard) model/system herein.

Define the Cahn-Hilliard energy functional
\begin{equation}\label{energy}
\Jc_\vep(\vphi):=\int_\Ome \Bigl[\frac{\vep}{2}|\nab\vphi|^2 
+ \frac{1}{\vep} F(\vphi) \Bigr]\, dx.
\end{equation}
Like many diffuse interface models (cf. \cite{am2,feng06,lt,llg02a,mcfadden}), 
the DCH system is also a dissipative system as it satisfies the following 
energy dissipation law (see Sec.~\ref{sec-2} for the details):
\begin{equation}\label{eq1.12}
\frac{d \Jc_\vep(\vphi)}{dt}+\vep\norm{\nab \mu}{L^2}^2+
\frac{1}{\gamma} \norm{\bu}{L^2}^2=0.
\end{equation}
As expected, the above energy law plays a vital role in the analysis
of the DCH system and in the design and analysis of numerical
methods for the system (see Secs.~\ref{sec-2}--\ref{sec-4} for
the details). 

This paper consists of four additional sections.
Section~\ref{sec-2} is devoted to the PDE analysis of 
the initial-boundary value problem \eqref{eq1.5}--\eqref{eq1.11}. 
Weak solutions are defined and the uniqueness and regularities 
of weak solutions are established. Section~\ref{sec-3} contains
the formulation of our fully discrete implicit finite element method for
problem \eqref{eq1.5}--\eqref{eq1.11}. 
It is shown that the proposed numerical method satisfies 
a mass conservation law in addition to a discrete energy law that 
mimics the basic energy law for the Darcy-Cahn-Hilliard phase 
field model and holds uniformly in the phase field parameter $\varepsilon$.
With help of the discrete energy law, it also proved that the fully discrete
finite method is unconditionally energy stable and uniquely solvable at
each time step. Section~\ref{sec-4} presents a convergence 
analysis for the proposed fully discrete finite element method.
Using the compactness method it is shown that the finite element 
solution has an accumulation point that is a weak 
solution of problem \eqref{eq1.5}--\eqref{eq1.11}. As a byproduct,
this convergence result also provides a constructive proof of the existence
global-in-time weak solutions to the PDE system \eqref{eq1.5}--\eqref{eq1.11} 
in both two and three dimensions. Finally, in Sec.~\ref{sec-5} we provide 
some results of numerical experiments validating our theoretical results 
and showing the effectiveness
of the proposed fully discrete finite element method.
To solve the nonlinear finite element equations at each time 
step, we propose a nonlinear multigrid iterative method to 
do the job. The details of the nonlinear multigrid solver 
and some other algorithmic and implementation issues 
are described in Appendix \ref{app-multigrid-solver}.

\section{PDE analysis}\label{sec-2}

The standard space notations are used in this paper, we refer to
\cite{adams,ciarlet} for their exact definitions. In particular,
$B^*$ denotes the dual space of a Banach space $B$, and $\bB$ denotes
the vector Banach space $B^d$, where $d$ is the dimension space.
Here we shall assume $d = 2$ or 3. The symbol $(\cdot,\cdot)$ is
used to denote the standard $L^2(\Ome)$ inner product,
$\dual{\cdot}{\cdot}$ stands for the dual product 
between $H^1(\Ome)$ and $(H^{1}(\Ome))^*$. $L^2_0(\Ome)$ denotes
the subspace of $L^2(\Ome)$ whose functions have zero mean.
Throughout the paper, unless stated otherwise,  $c$ and $C$ will be
used to denote generic positive constants which are independent of $p$, $\mu$, $\vphi$, $\bu$, and $\vep$.  If, for example, there is a dependence on $\vep$, 
we shall explicitly write $C=C(\vep)$.  As indicated earlier, 
we shall assume that $0<\vep<1$.

In the next section we shall construct a finite element method which
directly approximates variables $p$, $\mu$, and $\vphi$, but not $\bu$,
which will be computed as an auxiliary variable as needed. Specifically,
we shall approximate the pressure equation by a standard finite
element method and the phase equation by a mixed finite element
method. We remark that it is also a viable strategy that
approximates both the pressure equation and the phase equation
by mixed finite element methods, which we shall study in a separate work.  

To get the governing equations without using $\bu$, substituting the 
expression of $\bu$ in \eqref{eq1.5} into \eqref{eq1.6} and \eqref{eq1.7} we get
\begin{alignat}{2} \label{eq2.1}
\div\left(\nab p+\gamma \vphi\nab \mu\right) &=0 &&\qquad \mbox{in } \Ome_T, \\
\vphi_t-\vep\Del\mu-\div\left(\vphi\bigl[\nab p+\gamma\vphi\nab \mu\bigr]\right) 
&=0 &&\qquad\mbox{in } \Ome_T. \label{eq2.2}
\end{alignat}
Then the PDE system to be 
studied and approximated in this paper consists of equations
\eqref{eq2.1}, \eqref{eq2.2}, and \eqref{eq1.8}, which are
complimented with the boundary and initial conditions 
\eqref{eq1.9}--\eqref{eq1.11}.

Motivated by the energy law \eqref{eq1.12}, we define the following 
weak formulation and solutions to the initial-boundary value problem. 

\begin{definition}\label{weak_form}
Let $\vphi_0^\vep\in H^1(\Ome)$.
A triple $(p,\mu,\vphi)$ is called a weak solution of problem
\eqref{eq2.1}, \eqref{eq2.2}, and \eqref{eq1.8}--\eqref{eq1.11} if 
it satisfies  
	\begin{eqnarray}
p &\in& L^{\frac{4}{3}}\left((0,T);H^1(\Ome)\cap L^2_0(\Ome)\right),
	\\
 \mu &\in& L^2\bigl((0,T);H^1(\Ome)\bigr),
	\\
\nab p+\gamma \vphi\nab \mu &\in& L^2\bigl((0,T); \bL^2(\Ome)\bigr), 
	\\
\vphi &\in& L^\infty\left((0,T);H^1(\Ome)\right)\cap 
L^4\left((0,T);L^\infty(\Ome)\right),
	\\
\vphi_t &\in&  L^{\frac{4}{3}}\bigl((0,T); (H^1(\Ome))^*\bigr), 
	\end{eqnarray}
and there hold for almost all $t\in (0,T)$ 
\begin{alignat}{2}
\bigl(\nab p+\gamma\vphi \nab\mu,\nab q\bigr) &=0 
&&\qquad \forall q\in H^1(\Ome), \label{eq2.3} \\
\dual{\vphi_t}{\nu} +\vep\bigl(\nab \mu,\nab\nu \bigr) 
+\prodt{\vphi[\nab p+\gamma\vphi\nab \mu]}{\nab \nu} &= 0  
&&\qquad \forall \nu\in H^1(\Ome), \label{eq2.4} \\
\prodt{\mu}{\psi}-\vep \prodt{\nab\vphi}{\nab\psi}
-\frac{1}{\vep} \prodt{f(\vphi)}{\psi} &= 0 
&&\qquad \forall \psi\in H^1(\Ome), \label{eq2.5}
\end{alignat}
with the initial condition $\vphi(0)=\vphi_0^\vep$.  
\end{definition}

\begin{remark}\label{rem2.1}
The reason for not breaking the sum $\nab p+\gamma \vphi\nab\mu$
is that it has  better regularity/integrability than each of its
two terms. By the Aubin-Lions lemma (cf. \cite{simon85}),
the regularity on $\vphi$ ensures that $\vphi\in C^0([0,T]; L^2(\Ome))$.
Hence, the initial condition $\vphi(0)=\vphi_0^\vep$ makes sense.
\end{remark}
\begin{remark}
The regularities imposed on the solution $(p,\mu,\vphi)$ in the definition 
are {\em not} the ``minimum" required to make all terms in 
\eqref{eq2.3}--\eqref{eq2.5} be well defined. These regularities
are imposed because they are suggested by the energy law \eqref{eq1.12}.  
Moreover, the product space of the spaces used in the definition 
for $(p,\mu,\vphi)$ is indeed the energy space associated with the 
DCH system. 
\end{remark}

As we mentioned in Sec.~\ref{sec-1}, a key feature of 
the Darcy-Cahn-Hilliard (DCH) system is that it is a dissipative system in 
the sense that it satisfies an energy law, namely, \eqref{eq1.12}.
Below we demonstrate that this is indeed the case for weak solutions of 
the DCH system.   
 
\begin{lemma}\label{lem2.2}
Suppose that $\vphi_0^\vep\in H^1(\Ome)$ and $\Ome\subset\mathbb{R}^d$ $(d=2,3)$ 
is a Lipschitz domain, 
let $(p,\mu,\vphi)$ be a weak solution defined by \eqref{eq2.3}--\eqref{eq2.5}. 
In addition, suppose that the initial value $\vphi_0^\vep$ satisfies  
$\Jc_\vep\left(\vphi_0^\vep\right)\leq C_0$ for some $\vep$-independent  
constant $C_0$, \emph{i.e.}, the initial energy is uniformly bounded in $\vep$. Set 
$\bu:=-\left(\nab p+\gamma\vphi\nab\mu\right)$. Then, for almost all $t\in(0,T)$, 
	\begin{equation}
\int_\Ome \vphi(x,t)\,dx = \int_\Ome \vphi_0^\vep(x)\, dx, 
	\label{eq2.6} 
	\end{equation}
and for all $t\in (0,T)$ and some $\vep$-independent constant $C=C(E(0))>0$
	\begin{align}
&E(t) + \int_0^t \Bigl[\vep\norm{\nab\mu(s)}{L^2}^2 
+ \frac{1}{\gamma}\norm{\bu(s)}{L^2}^2 \Bigr]\, ds = E(0)<\infty, 
	\label{eq2.7} 
	\\
&\max_{0\le s\le t} \norm{\vphi(s)}{H^1}^2 \le \frac{C}{\vep},
	\label{LinfH1-phi}
	\\
&\int_0^t\norm{\mu(s)}{H^1}^2\, ds \le \frac{(T+1)C}{\vep^5} ,
	\label{L2H1-mu}
	\\
&\int_0^t\norm{\mu(s)-\vep^{-1}f\left(\vphi(s)\right)}{L^2}^2\, ds \le \frac{(T+1)C}{\vep^5} ,
	\label{L2L2-g}
	\\
&\int_0^t \norm{\vphi_t(s)+ \bu(s)\cdot\nab\vphi(s)}{\left(H^1\right)^*}^2\, ds \leq C\vep, 
	\label{eq2.7a} 
	\\
&\int_0^t \norm{\vphi_t(s)}{(W^{1,3})^*}^2\, ds \leq \frac{C}{\vep},  \label{eq2.7b}\\
&\max_{0\le s\le t} \norm{ |\vphi(s)|-1 }{L^2}^2 \leq C\vep, \label{eq2.7c}
\end{align}
where  $E(t):=\Jc_\vep\left(\vphi(t)\right)$ and $\Jc_\vep(\, \cdot\, )$ 
is defined in \eqref{energy}.
\end{lemma}

\begin{proof}
\eqref{eq2.6} follows trivially from setting $\nu=1$ in \eqref{eq2.4}.
To prove \eqref{eq2.7}, we consider two cases separately. First, suppose that 
$\vphi_t\in L^2\left((0,T); H^1(\Ome)\right)$.  
Then setting $q=\frac{p}{\gamma}$ in \eqref{eq2.3},
$\nu=\mu$ in \eqref{eq2.4} and $\psi=-\vphi_t$ in \eqref{eq2.5} 
(note that by the assumption, $-\vphi_t$ is a valid test function), and
adding the resulting equations we get
\[
\frac{d}{dt} \Bigl[ \vep\norm{\nab\vphi}{L^2}^2 
+\frac{1}{\vep} \bigl(F(\vphi), 1\bigr) \Bigr]+\vep\norm{\nab\mu}{L^2}^2 
+ \frac{1}{\gamma}\norm{\nab p +\gamma\vphi\nab\mu}{L^2}^2 =0 .
\]
Integrating over the interval $(0,t)$ yields \eqref{eq2.7}.

For the general case 
$\vphi_t\in L^{\frac{4}{3}}\big((0,T);(H^1(\Ome))^*\bigr)$, 
and we note that $\psi=-\vphi_t$ is not a valid test function 
in \eqref{eq2.5}. However, this technical difficulty can be overcome 
by using a Steklov average technique.
For $t\in (0,T)$, let $\delta>0$ be a small number. Define the 
Steklov average $\vphi^\delta$ of $\vphi$ by (cf. \cite[Ch.~2]{LSU})
\begin{alignat*}{2}
\vphi^\delta(\, \cdot\, , t) :=S_{+}^\delta (\vphi) (\, \cdot\, ,t) &:=
\frac{1}{\delta} \int_t^{t+\delta} \vphi(\, \cdot\, ,s)\,ds &\qquad\forall t
\in (0,T),
\end{alignat*}
Trivially, for small enough $\delta$,
\[
\vphi^\delta_t(\,\cdot\, , t) :=\bigl(\vphi^\delta(\, \cdot,\,  t)\bigr)_t
= \frac{\vphi(\, \cdot\, ,t+\delta)-\vphi(\, \cdot\, ,t)}{\delta}.
\]
Hence, $\vphi^\delta_t(\, \cdot\, ,t) \in H^1(\Ome)$ for almost every
$t\in (0,T-\delta)$.  It is well known that (cf. \cite[Ch.~2]{LSU})
\begin{equation}\label{e2.9b}
S_{+}^\delta (\vphi_t)=\left(S_{+}^\delta (\vphi)\right)_t
=\vphi^\delta_t.
\end{equation}
Note that the derivative on the left-hand side of the above identity
is understood as a distributional derivative, while the derivative
on the right-hand side is understood in the classical sense.

Now applying $S_+^\delta$ to both sides of \eqref{eq2.3}--\eqref{eq2.5}
after replacing $t$ by $s$ yields
\begin{alignat}{2}
\bigl(\nab p^\delta+\gamma (\vphi \nab\mu)^\delta,\nab q\bigr) &=0 
&&\qquad \forall q\in H^1(\Ome), \label{eq2.3x} \\
(\vphi^\delta_t,\nu) +\vep\bigl(\nab \mu^\delta,\nab\nu \bigr) 
+\prodt{(\vphi[\nab p+\gamma\vphi\nab \mu])^\delta}{\nab \nu} &= 0  
&&\qquad \forall \nu\in H^1(\Ome), \label{eq2.4x} \\
\prodt{\mu^\delta}{\psi}-\vep \prodt{\nab\vphi^\delta}{\nab\psi}
-\frac{1}{\vep} \prodt{(f(\vphi))^\delta}{\psi} &= 0 
&&\qquad \forall \psi\in H^1(\Ome). \label{eq2.5x}
\end{alignat}
Setting $\psi=-\vphi^\delta_t$ in \eqref{eq2.5x}, $\nu=\mu^\delta$ in
\eqref{eq2.4x} and $q=\frac{p^\delta}{\gamma}$ in \eqref{eq2.3x} and 
adding the resulting equations we get
\begin{align}\label{eq2.5y}
&\frac{d}{dt} \mathcal{J}_\vep(\vphi^\delta)
+\vep\norm{\nab\mu^\delta}{L^2}^2 
+\frac{1}{\gamma}\norm{\nab p^\delta +\gamma(\vphi\nab\mu)^\delta}{L^2}^2 
=\mathcal{R}^\delta(t),
\end{align}
where
\begin{align*}
\mathcal{R}^\delta(t)&:=\frac{1}{\vep} \prodt{f(\vphi^\delta)
-(f(\vphi))^\delta}{\vphi^\delta_t} 
+\prodt{ \nab p^\delta +\gamma(\vphi\nab\mu)^\delta }
{\vphi\nab\mu^\delta-(\vphi\nab\mu)^\delta }\\
&\qquad
+\prodt{ \vphi[\nab p^\delta +\gamma(\vphi\nab\mu)^\delta] 
-(\vphi[\nab p +\gamma\vphi\nab\mu])^\delta}{\nab \mu^\delta}. 
\end{align*}

Integrating \eqref{eq2.5y} in $t$ gives 
\begin{align*}
\mathcal{J}_\vep(\vphi^\delta(s)) 
&+\int_0^s\Bigl(\vep\norm{\nab\mu^\delta(t)}{L^2}^2 
+\frac{1}{\gamma}\norm{\nab p^\delta(t)
+\gamma(\vphi\nab\mu)^\delta(t)}{L^2}^2 \Bigr)\,dt  \\
&= \mathcal{J}_\vep(\vphi^\delta(0))+\int_0^s \mathcal{R}^\delta(t)\, dt
\qquad\forall s\in (0,T).
\end{align*}
Note that, for each fixed $\vep>0$, $f(\vphi)\in L^2((0,T);H^1(\Ome))$,  since 
$\vphi\in L^4((0,T);L^\infty(\Ome))$,  and $f(\, \cdot\, )$ is continuous.  
Sending $\delta\to 0^+$ and using properties of the Steklov 
average $S^\delta_+$ (cf. \cite[Ch.~2]{LSU}) we get
\begin{align*}
&\lim_{\delta\to 0^+} \int_0^s \mathcal{R}^\delta(t)\, dt =0,\\
\mathcal{J}_\vep(\vphi(s)) +\int_0^s\Bigl(\vep\norm{\nab\mu(t)}{L^2}^2 
&+\frac{1}{\gamma}\norm{\nab p(t)
+\gamma(\vphi(t)\nab\mu(t))}{L^2}^2 \Bigr)\,dt = \mathcal{J}_\vep(\vphi(0)).
\end{align*}
Hence, we recover \eqref{eq2.7}. Observe that from \eqref{eq2.7} we can 
conclude that $E(t)$ is an absolutely continuous function of time.

Using \eqref{eq2.7} and the estimate
	\begin{equation}
\left(F(\vphi),1\right) \ge \frac{1}{2}\norm{\vphi}{L^2}^2 - \frac{3}{4}\left|\Omega\right| 
	\end{equation}
we discover that $\norm{\phi}{L^2}^2 \le C$, for all time and independent 
of $\vep$, and inequality~\eqref{LinfH1-phi} follows.	 Inequalities~\eqref{L2H1-mu} and \eqref{L2L2-g} follow straightforwardly from \eqref{eq2.5}, 
\eqref{LinfH1-phi}, and the Sobolev embedding 
$H^1(\Ome)\hookrightarrow L^6(\Ome)$ for $d=2,\, 3$.  Inequality 
\eqref{eq2.7a} 
is an immediate consequence of \eqref{eq2.4} and the 
fact that $\vep\norm{\nab \mu}{L^2\left(L^2\right)}< C <\infty$ from 
\eqref{eq2.7}. 
To show \eqref{eq2.7b}, by \eqref{eq2.4},  
the Sobolev embedding $H^1(\Ome)\hookrightarrow L^6(\Ome)$ for $d=2,\, 3$,
and \eqref{LinfH1-phi},
we get for any $\nu\in W^{1,3}(\Ome)$
\begin{align*}
\dual{\vphi_t}{\nu}
&=-\vep\bigl(\nab\mu,\nab\nu\bigr)+\prodt{\vphi\bu}{\nab \nu}
	\\
&\leq\vep\norm{\nab\mu}{L^2} \norm{\nab\nu}{L^2} 
+\norm{\vphi}{L^6} \norm{\bu}{L^2} \norm{\nab \nu}{L^3} 
	\\
&\leq C\bigl[\vep \|\nab\mu\|_{L^2} + \norm{\vphi}{H^1} \|\bu\|_{L^2}\bigr]
\, \|\nab\nu\|_{L^3}
	\\
&\leq C\left[\vep \|\nab\mu\|_{L^2} + \frac{1}{\sqrt{\vep}} \|\bu\|_{L^2}\right]
\, \|\nab\nu\|_{L^3} .
\end{align*}
The above inequality and \eqref{eq2.7} infer \eqref{eq2.7b}.  

Finally, \eqref{eq2.7c} is an immediate consequence of 
\eqref{eq2.7} and the inequality $(\vphi^2-1)^2\geq \bigl(|\vphi|-1\bigr)^2$. 
The proof is complete.  
\end{proof}

The next lemma shows that weak solutions have some additional
regularities.

\begin{lemma}\label{lem2.1}
Suppose that $\vphi_0^\vep\in H^1(\Ome)$ and $\Ome\subset\mathbb{R}^d$ 
$(d=2,3)$ is a Lipschitz domain. Let $(p,\mu,\vphi)$ be a weak solution 
defined by \eqref{eq2.3}--\eqref{eq2.5}. Then 
$\vphi\in L^2\left(\left(0,T\right);H^2\left(\Ome\right)\right)$. 
Moreover, if $\Ome$ is a convex polygonal or polyhedral domain, 
then $\vphi\in L^2\left(\left(0,T\right);H^3\left(\Ome\right)\right)$.
\end{lemma}

\begin{proof}
We begin by rewriting~\eqref{eq2.5} as
\begin{equation}
\label{eq2.5a}
\vep\bigl(\nab\vphi,\nab\psi\bigr) = \bigl(\mu-\vep^{-1}f(\vphi),\psi\bigr) 
\qquad \forall \psi\in H^1(\Ome).  
\end{equation}
Hence, $\vphi$ is a weak solution to a Poisson equation with homogeneous
Neumann boundary conditions and the right-hand side ``source" function
$g:=\mu-\vep^{-1}f(\vphi)$.  Since $\vphi\in L^4
\left(\left(0,T\right);L^\infty\left(\Ome\right)\right)$, then
$g\in L^2 \left(\left(0,T\right);H^1\left(\Ome \right)\right)$. 
By elliptic regularity theory (cf. \cite[Ch.~7]{GT}) we conclude 
that $\vphi\in L^2\left(\left(0,T\right);H^2\left(\Ome\right)\right)$ 
and $\frac{\p \vphi}{\p \bn}=0$ on $\p\Ome$ in the distributional sense.

Introduce the function space 
\begin{displaymath}
U:=\Bigl\{\theta\in H^2(\Ome);\,\frac{\p\theta}{\p \bn}=0 
\,\,\mbox{on } \p\Ome\Bigr\}.
\end{displaymath}
For any $\theta\in U\cap C^\infty(\Ome)$, setting $\psi=\Del \theta$ 
in \eqref{eq2.5a}, integrating by parts and using a density argument we get
\begin{displaymath}
\vep\bigl(\Del\vphi,\Del\theta\bigr) 
=\left(\nab\left[\mu-\vep^{-1}f(\vphi)\right],\nab\theta\right) 
\qquad \forall\, \theta\in U,
\end{displaymath}
which implies that $\vphi\in H^2(\Ome)$ with $(\vphi,1)=\mbox{const}$
is the unique weak solution to the following biharmonic problem:
\begin{alignat*}{2}
\Del^2 \vphi &= -\Del g &&\qquad\mbox{in }\Ome, \\
\frac{\p \vphi}{\p \bn} =\frac{\p \Delta\vphi}{\p \bn}
&=0 &&\qquad\mbox{on }\p\Ome. 
\end{alignat*}
Since $-\Del g \in L^2\left(\left(0,T\right);H^{-1}\left(\Ome\right)\right)$, 
it follows from a well-known regularity result of \cite{blum} 
that $\vphi\in L^2\left(\left(0,T\right);
H^3\left(\Ome\right)\right)$ when $\Ome$ is a convex polygonal domain. 
The proof is complete.
\end{proof}

We conclude this section by establishing the following uniqueness 
theorem for weak solutions of problem \eqref{eq2.1}, \eqref{eq2.2},  
and \eqref{eq1.8}--\eqref{eq1.11} defined in Definition~\ref{weak_form}.

\begin{theorem}\label{thm2.1}
Suppose that $\vphi_0^\vep\in H^1(\Ome)$ with 
$\Jc_\vep\left(\vphi_0^\vep\right)\leq C_0$ for some $\vep$-independent 
constant $C_0$ and $\Ome\subset\mathbb{R}^d$ $(d=2,3)$ is a 
Lipschitz domain. We say that a weak solution $(p,\mu,\vphi)$ 
belongs to the function space $\mathcal{F}$ if it satisfies the additional 
regularity conditions $\nab p+\gamma\vphi \nab \mu \in L^{\frac{12}{6-d}}((0,T); 
L^2(\Ome))$, $\mu\in L^{\frac{12}{6-d}}((0,T); H^1(\Ome))$, and
$\vphi_t\in L^2((0,T); (H^1(\Ome))^*)$.
Then weak solutions of \eqref{eq2.3}--\eqref{eq2.5} in the function class 
$\mathcal{F}$ are unique.
\end{theorem}

\begin{proof}
Since the proof is long, we divide it into six steps.

{\em Step 1:}
Suppose $(p_i,\mu_i,\vphi_i)$, $i=1,\, 2$, are two weak solutions,
and define $\bu_i:=-\nab p_i-\gamma\vphi_i\nab \mu_i$, $i=1,\, 2$.
Let $p=p_1-p_2$, $\bu=\bu_1-\bu_2$,  $\mu=\mu_1-\mu_2$, and , $\vphi=\vphi_1-\vphi_2$. Subtracting the corresponding equations
of \eqref{eq2.3}--\eqref{eq2.5} satisfied by
$\left(p_1,\mu_1,\vphi_1\right)$ and $\left(p_2,\mu_2,\vphi_2\right)$
we get the following ``error" equations:
\begin{alignat}{2}
\bigl(\bu,\nab q\bigr) &=0 &&\qquad \forall q\in H^1(\Ome),\label{eq2.100} \\
\dual{\vphi_t}{\nu} +\vep \bigl(\nab \mu,\nab\nu \bigr)
-\prodt{\vphi_1\bu+\vphi\bu_2}{\nab \nu} &= 0
&&\qquad \forall \nu\in H^1(\Ome), \label{eq2.101} \\
\prodt{\mu}{\psi}-\vep \prodt{\nab\vphi}{\nab\psi}
-\frac{1}{\vep} \prodt{f(\vphi_1)-f(\vphi_2)}{\psi} &= 0
&&\qquad \forall \psi\in H^1(\Ome). \label{eq2.102}
\end{alignat}
We will frequently use the fact that $\int_\Ome \vphi(x,t)\,dx =0$, for 
almost all $t \in (0,T)$, which follows from \eqref{eq2.6}.

Setting $\nu=\vphi$ in \eqref{eq2.101} and  $\psi=\mu$ in \eqref{eq2.102},
adding the resulting equations, and using the fact that
$\left(\bu_2,\nab\left( \vphi^2\right)\right)=0
=\left(\bu,\nab (\vphi_1\vphi)\right)$ we get
\begin{align}\label{eq2.103}
\frac12 \frac{d}{dt}\norm{\vphi}{L^2}^2 + \norm{\mu}{L^2}^2  
=-\bigl(\bu\cdot\nab \vphi_1, \vphi\bigr) 
+\frac{1}{\vep} \bigl(g(\vphi_1,\vphi_2)\vphi, \mu\bigr),
\end{align}
where $g(\vphi_1,\vphi_2)=\vphi_1^2+\vphi_1\vphi_2+\vphi_2^2 -1$.

Using Schwarz inequality and the following Gagliardo-Nirenberg
inequality (cf. \cite{adams,elliott86})
\[
\|\vphi\|_{L^\infty} \leq C\|\Delta\vphi\|_{L^2}^{\frac{d}{4}} 
\|\vphi\|_{L^2}^{\frac{4-d}{4}}+C\|\vphi\|_{L^2} \qquad (d=2,3)
\]
in \eqref{eq2.103} we get
\begin{align*}
&\frac{d}{dt}\norm{\vphi}{L^2}^2 + 2\norm{\mu}{L^2}^2  
\leq 2 \|\bu\|_{L^2} \|\nab\vphi_1\|_{L^2} \|\vphi\|_{L^\infty} 
+\frac{2}{\vep}\|g(\vphi_1,\vphi_2)\|_{L^\infty}\norm{\vphi}{L^2}\norm{\mu}{L^2}\\ 
&\quad
\leq \frac{\vep}{4\gamma} \|\bu\|_{L^2}^2 +\frac{C}{\vep}\|\nab\vphi_1\|_{L^2}^2 
\|\Delta\vphi\|_{L^2}^{\frac{d}{2}} \|\vphi\|_{L^2}^{\frac{4-d}{2}}
+ \frac{2}{\vep}\|g(\vphi_1,\vphi_2)\|_{L^\infty}\norm{\vphi}{L^2}\norm{\mu}{L^2}\\
&\quad
\leq \frac{\vep}{4\gamma}\|\bu\|_{L^2}^2+\frac{\vep^2}{16}\|\Delta\vphi\|_{L^2}^2 
+C(\vep)\|\nab\vphi_1\|_{L^2}^{\frac{8}{4-d}}
\norm{\vphi}{L^2}^2 +\norm{\mu}{L^2}^2 \\
&\qquad
+\frac{1}{\vep^2}\|g(\vphi_1,\vphi_2)\|_{L^\infty}^2\norm{\vphi}{L^2}^2.
\end{align*}
Hence, it follows from \eqref{eq2.7} that
\begin{align} \label{eq2.103x}
\frac{d}{dt}\norm{\vphi}{L^2}^2 + \norm{\mu}{L^2}^2  
&\leq \frac{\vep}{4\gamma} \|\bu\|_{L^2}^2 
+ \frac{\vep^2}{16} \|\Delta\vphi\|_{L^2}^2 \\
&\hskip 0.6in
+C(\vep)\Bigl(1
+\|g(\vphi_1,\vphi_2)\|_{L^\infty}^2\Bigr) \norm{\vphi}{L^2}^2. 
\nonumber
\end{align}

\medskip
{\em Step 2:}
Setting $\psi=\Delta \vphi$ in \eqref{eq2.102} gives
\begin{align*}
\vep\|\Delta\vphi\|_{L^2}^2
&=-(\mu,\Delta\vphi)+\frac{1}{\vep}\bigl(g(\vphi_1,\vphi_2)\vphi,\Delta\vphi\bigr) \\
&\leq \frac{\vep}2 \|\Delta\vphi\|_{L^2}^2 +\frac{1}{\vep} \|\mu\|_{L^2}^2
+ \frac{1}{\vep^3} \|g(\vphi_1,\vphi_2)\|_{L^\infty}^2 \|\vphi\|_{L^2}^2.
\end{align*}
Hence
\begin{align}\label{eq2.105x}
\frac{\vep^2}4\|\Delta\vphi\|_{L^2}^2
&\leq \frac12 \|\mu\|_{L^2}^2
+ C(\vep) \norm{g(\vphi_1,\vphi_2)}{L^\infty}^2 \|\vphi\|_{L^2}^2.
\end{align}

\medskip
{\em Step 3:} First, on noting that
\begin{align*}
\bu &=\bu_1-\bu_2 
=-\nab p - \gamma \bigl( \vphi_1\nab \mu_1-\vphi_2\nab \mu_2 \bigr)
=-\nab p - \gamma \vphi_1 \nab\mu -\gamma \vphi \nab\mu_2 ,
\end{align*}
and using \eqref{eq2.100} with $q=p$ we obtain
\begin{align}\label{eq2.104}
\frac{1}{\gamma} \norm{\bu}{L^2}^2 &=-\frac{1}{\gamma}\bigl(\bu,\nab p\bigl)
- \bigl(\bu, \vphi_1 \nab\mu \bigr) 
- \bigl(\bu, \vphi \nab\mu_2 \bigr) \\
&= - \bigl(\vphi_1\bu, \nab\mu \bigr) 
- \bigl(\vphi\bu, \nab\mu_2 \bigr).   \nonumber
\end{align}

Second, setting $\nu=\mu$ in \eqref{eq2.101} yields
\begin{align}\label{eq2.104x}
\langle \vphi_t,\mu\rangle+\vep \|\nab \mu\|_{L^2}^2 
&=\bigl( \vphi_1\bu+ \vphi\bu_2, \nab \mu\bigr).
\end{align}

Third, applying the Steklov average operator $S_+^\delta$ to \eqref{eq2.102} 
(we use the same notion as in the proof of Lemma~\ref{lem2.2}) yields
\[
\prodt{\mu^\delta}{\psi}-\vep \prodt{\nab\vphi^\delta}{\nab\psi}
-\frac{1}{\vep} \prodt{(g(\vphi_1,\vphi_2)\vphi)^\delta}{\psi} = 0
\qquad \forall \psi\in H^1(\Ome). 
\]
Where $g(\vphi_1,\vphi_2)=\vphi_1^2+\vphi_1\vphi_2+\vphi_2^2-1$.
Setting $\psi=-\vphi^\delta_t$ in the above equation gives
\begin{align*}
-\prodt{\mu^\delta}{\vphi^\delta_t}
+\frac{\vep}2 \frac{d}{dt}\|\nab\vphi^\delta\|_{L^2}^2
+\frac{1}{\vep} \prodt{(g(\vphi_1,\vphi_2)\vphi)^\delta}{\vphi^\delta_t} = 0.
\end{align*}
Taking the limit $\delta\to 0^+$ and using the properties of
Steklov average operator (cf. \cite{LSU}) we get
\begin{align}\label{eq2.106}
-\dual{\vphi_t}{\mu} +\frac{\vep}2 \frac{d}{dt}\|\nab\vphi\|_{L^2}^2
+\frac{1}{\vep} \dual{\vphi_t}{g(\vphi_1,\vphi_2)\vphi} = 0.
\end{align}

Finally, adding \eqref{eq2.104}, \eqref{eq2.104x}, and \eqref{eq2.106},
using the fact that $\bigl(\bu_2, \nab(\vphi\mu)\bigr) 
=\bigl(\bu, \nab(\vphi\mu_2)\bigr) =0$ and Young's and 
Gagliardo-Nirenberg inequalities (cf. \cite{adams,elliott86}) we get
\begin{align*}
\frac{1}{\gamma}\norm{\bu}{L^2}^2 &+ \vep \|\nab \mu\|_{L^2}^2 
+\frac{\vep}2 \frac{d}{dt}\|\nab\vphi\|_{L^2}^2
+\frac{1}{\vep} \dual{\vphi_t}{g(\vphi_1,\vphi_2)\vphi}
	\\ 
&=\bigl( \vphi\bu_2, \nab \mu\bigr) - \bigl(\vphi\bu, \nab\mu_2 \bigr)
=-\bigl(\bu_2\cdot\nab \vphi,\mu\bigr)+\bigl(\bu\cdot\nab\vphi,\mu_2 \bigr)
	\nonumber
	\\
&\leq \|\bu_2\|_{L^2} \|\nab\vphi\|_{L^3} \|\mu\|_{L^6}
+ \|\bu\|_{L^2} \|\nab\vphi\|_{L^3} \|\mu_2\|_{L^6}  
	\nonumber
	\\
&\leq C \bigl( \|\bu_2\|_{L^2} \|\mu\|_{H^1}
+ \|\bu\|_{L^2}  \|\mu_2\|_{H^1}\bigr)\bigl(\|\Delta \vphi\|_{L^2}^{\frac{d}{6}}
\|\nab \vphi\|_{L^2}^{\frac6{6-d}} + \|\nab \vphi\|_{L^2}\bigr)
	\nonumber
	\\
&\leq \frac{1}{4\gamma} \norm{\bu}{L^2}^2 + \frac{\vep}2 \norm{\mu}{H^1}^2
+\frac{\vep}{16} \|\Delta \vphi\|_{L^2}^2 
	\nonumber
	\\
&\hskip 0.5in 
+C(\vep)\left(\norm{\bu_2}{L^2}^{\frac{12}{6-d}}
+\norm{\mu_2}{H^1}^{\frac{12}{6-d}}\right) \|\nab \vphi\|_{L^2}^2.
	\nonumber
\end{align*}
Hence
\begin{align}\label{eq2.111}
\frac{3}{4\gamma}\norm{\bu}{L^2}^2 &+ \frac{\vep}2 \|\nab \mu\|_{L^2}^2 
+\frac{\vep}2 \frac{d}{dt}\|\nab\vphi\|_{L^2}^2
+\frac{1}{\vep} \dual{\vphi_t}{g(\vphi_1,\vphi_2)\vphi}
	\\ 
&\leq \frac{\vep}{16} \|\Delta \vphi\|_{L^2}^2 +\frac{\vep}{2}\norm{\mu}{L^2}^2 
+C(\vep)\left(\norm{\bu_2}{L^2}^{\frac{12}{6-d}}
+\norm{\mu_2}{H^1}^{\frac{12}{6-d}}\right) \|\nab \vphi\|_{L^2}^2 .
	\nonumber
\end{align}

\medskip
{\em Step 4}:  
To control the last term on the left-hand side of \eqref{eq2.111}, we rewrite
\[
g(\vphi_1,\vphi_2)= \vphi_1^2 +\vphi_1\vphi_2+\vphi_2^2 -1
=(\vphi_1-\vphi_2)^2 +3 \vphi_1\vphi_2 -1
=(\vphi^2-1) + 3 \vphi_1\vphi_2.
\]
Thus
\begin{align} \label{eq2.112}
& \dual{\vphi_t}{g(\vphi_1,\vphi_2)\vphi}
=\frac{1}{2} \dual{(\vphi^2)_t}{\vphi^2-1} 
   +\frac{3}{2} \dual{(\vphi^2)_t}{\vphi_1\vphi_2} \\
&\hskip 0.1in
=\frac{d}{dt}\Bigl[ \frac{1}{4} \|\vphi\|_{L^4}^4 
  -\frac{1}{2} \|\vphi\|_{L^2}^2
  +\frac{3}{2} \prodt{\vphi^2}{\vphi_1\vphi_2} \Bigr]
  -\frac{3}{2} \dual{(\vphi_1\vphi_2)_t}{\vphi^2} \nonumber \\
&\hskip 0.1in
= \frac{d}{dt}\Bigl[ \frac{1}{4} \|\vphi\|_{L^4}^4 
  -\frac{1}{2} \|\vphi\|_{L^2}^2
  +\frac{3}{2} \prodt{\vphi^2}{\vphi_1\vphi_2} \Bigr]
  -\frac32 \Bigl[ \dual{\vphi_{1t}}{\vphi_2\vphi^2}
  + \dual{\vphi_{2t}}{\vphi_1\vphi^2} \Bigr].  \nonumber
\end{align}

To bound the last term on the right-hand side of \eqref{eq2.112},
we introduce the inverse Laplace operator 
$\Delta^{-1}:\,(H^1(\Ome))^*\to H^1(\Ome)$. For any $w\in (H^1(\Ome))^*$
with $\langle w, 1\rangle=0$,
let $\Delta^{-1}w\in H^1(\Ome)$ be the unique solution of the following problem:   
\begin{alignat}{2}\label{eq2.113a}
\prodt{\nab (\Delta^{-1}w)}{\nab \eta} &=-\dual{w}{\eta} 
&&\qquad \forall \eta\in H^1(\Ome),\\
\prodt{\Delta^{-1}w}{1} &=0. && \label{eq2.113b}
\end{alignat}
It is straightforward to show that, for all $w\in (H^1(\Ome))^*$
with $\langle w, 1\rangle=0$,
	\begin{equation}
\norm{w}{\left(H^1\right)^*}  =  \norm{\nabla\left(\Delta^{-1}w \right)}{L^2} .
	\end{equation}

Then for $j,k=1,2$ and $j\neq k$ using Sobolev inequality
(cf. \cite{adams}) we have
\begin{align}\label{eq2.114}
&\dual{\vphi_{jt}}{\vphi_k\vphi^2} 
=-\prodt{\nab (\Delta^{-1}\vphi_{jt})}{\vphi^2\nab\vphi_k
+2\vphi_k\vphi\nab\vphi} 
	\\ 
&\quad
\leq \|\nab (\Delta^{-1}\vphi_{jt})\|_{L^2} \|\vphi^2\nab\vphi_k
+2\vphi_k\vphi\nab\vphi\|_{L^2} 
	\nonumber
	\\
&\quad
\leq \norm{\vphi_{jt}}{\left(H^1\right)^*} 
\Bigl( \|\nab\vphi_k\|_{L^6} \|\vphi\|_{L^6}^2
+ 2\|\vphi_k\|_{L^6} \|\vphi\|_{L^6} \|\nab\vphi\|_{L^6} \Bigr) 
	\nonumber
	\\
&\quad
\leq C\norm{\vphi_{jt}}{\left(H^1\right)^*} 
\Bigl(\|\Delta\vphi_k\|_{L^2} \|\nab \vphi\|_{L^2}^2
+\norm{\vphi_k}{H^1} \|\nab\vphi\|_{L^2} \|\Delta\vphi\|_{L^2} \Bigr) 
	\nonumber
	\\
&\quad
\leq \frac{\vep^2}{48} \|\Delta\vphi\|_{L^2}^2 
+ C(\vep) \Bigl( \norm{\vphi_{jt}}{\left(H^1\right)^*}\|\Delta\vphi_k\|_{L^2}
+\norm{\vphi_{jt}}{\left(H^1\right)^*}^2 \norm{\vphi_k}{H^1}^2 \Bigr)
\|\nab \vphi\|_{L^2}^2. \nonumber
\end{align}

\medskip
{\em Step 5}:
Adding \eqref{eq2.103x},  \eqref{eq2.105x} and $\vep$ times of
\eqref{eq2.111}, and utilizing \eqref{eq2.112} and \eqref{eq2.114} we get
\begin{align} \label{eq2.118}
&\frac{d}{dt}\Bigl[ \frac12\norm{\vphi}{L^2}^2
+\frac{\vep^2}2\|\nab\vphi\|_{L^2}^2
+\frac{1}{4} \|\vphi\|_{L^4}^4
+ \frac{3}{2} \prodt{\vphi^2}{\vphi_1\vphi_2} \Bigr] \\
&\hskip 0.5in
+\frac12(1-\vep^2) \norm{\mu}{L^2}^2 + \frac{\vep}{2\gamma} \norm{\bu}{L^2}^2
+\frac{\vep^2}2 \|\nab \mu\|_{L^2}^2 
+\frac{\vep^2}{16}\|\Delta\vphi\|_{L^2}^2 
	\nonumber
	\\
&\hskip 0.1in
\leq C(\vep)\Bigl(1+\norm{g(\vphi_1,\vphi_2)}{L^\infty}^2\Bigr) \norm{\vphi}{L^2}^2
	\nonumber
	\\
&\hskip 0.5in
+C(\vep) \Biggl[ \left(\norm{\bu_2}{L^2}^{\frac{12}{6-d}} 
+ \norm{\mu_2}{H^1}^{\frac{12}{6-d}}\right) 
+\sum_{j,k=1\atop j\neq k}^2 \Bigl\{\norm{\vphi_{jt}}{\left(H^1\right)^*}
\|\Delta\vphi_k\|_{L^2} 
	\nonumber
	\\ 
&\hskip 0.5in
+\norm{\vphi_{jt}}{\left(H^1\right)^*}^2 \norm{\vphi_k}{H^1}^2\Bigr\} \Biggr]\, \|\nab \vphi\|_{L^2}^2
	\nonumber
	\\
&\hskip 0.1in
\leq  a(t) \left( \norm{\vphi}{L^2}^2+ \|\nab \vphi\|_{L^2}^2 \right), 
	\nonumber
\end{align}
where
\begin{align*}
a(t)&:=C(\vep)\Biggl[1
+\norm{g(\vphi_1,\vphi_2)}{L^\infty}^2
+ \norm{\bu_2}{L^2}^{\frac{12}{6-d}} + \norm{\mu_2}{H^1}^{\frac{12}{6-d}}
 \\
&\qquad
+\sum_{j,k=1\atop j\neq k}^2 \Bigl\{\norm{\vphi_{jt}}{\left(H^1\right)^*}
\|\Delta\vphi_k\|_{L^2} 
+\norm{\vphi_{jt}}{\left(H^1\right)^*}^2 \norm{\vphi_k}{H^1}^2\Bigr\} \Biggr] .
\end{align*}

Integrating \eqref{eq2.118} in $t$ over the interval $(0,t)$ we get
\begin{align} \label{eq2.119}
&\norm{\vphi(t)}{L^2}^2 +\vep^2\|\nab\vphi(t)\|_{L^2}^2
+\|\vphi(t)\|_{L^4}^4 + 6\prodt{\vphi^2(t)}{\vphi_1(t)\vphi_2(t)} \\
&\hskip 0.6in
\leq \int_0^t a(s)\left( \norm{\vphi(s)}{L^2}^2
+\|\nab \vphi(s)\|_{L^2}^2 \right)\, dt. \nonumber
\end{align}

\medskip
{\em Step 6}: Define 
\[
\tau:=\max\Bigl\{t\in [0,T];\,\, \prodt{\vphi^2(s)}{\vphi_1(s)\vphi_2(s)}\geq 0
\,\, \forall s\in [0,t] \Bigr\}.
\]
We now show that $\tau>0$. Since $\vphi_1(0)=\vphi_2(0)=\vphi_0^\vep$, by
continuity there exists $t_1>0$ such that for $j=1,2$
\[
\Bigl(|\vphi_0^\vep|^2 -\frac{1}{\sqrt{12}},\eta\Bigr) 
\leq \bigl(\vphi_1(t) \vphi_2(t),\eta\bigr) 
\leq \Bigl(|\vphi_0^\vep|^2 +\frac{1}{\sqrt{12}},\eta\Bigr)
\quad \forall t\in [0,t_1], \quad \forall\, 0\leq \eta \in L^\infty(\Omega).
\]
Consequently,
\begin{align}\label{eq2.120}
\prodt{\vphi^2(t)}{\vphi_1(t)\vphi_2(s)}
\geq \|\vphi \vphi_0^\vep\|_{L^2}^2 -\frac{1}{12} \|\vphi\|_{L^2}^2
\quad \forall t\in [0,t_1].
\end{align}

Substituting \eqref{eq2.120} into \eqref{eq2.119} yields
\begin{align} \label{eq2.121}
& \frac{1}{2}\norm{\vphi(t)}{L^2}^2 +\vep^2\|\nab\vphi(t)\|_{L^2}^2
+\|\vphi(t)\|_{L^4}^4 + 6\|\vphi \vphi_0^\vep\|_{L^2}^2 \\
&\hskip 0.6in
\leq \int_0^t a(s)\left( \norm{\vphi(s)}{L^2}^2
+\|\nab \vphi(s)\|_{L^2}^2 \right)\, ds\qquad \forall t\in [0,t_1]. \nonumber
\end{align}
By Gronwall's inequality we get 
\begin{align} \label{eq2.122}
\frac{1}{2}\norm{\vphi(t)}{L^2}^2 +\vep^2\|\nab\vphi(t)\|_{L^2}^2
&\leq \left[ \frac{1}{2}\norm{\vphi(0)}{L^2}^2 +\vep^2\|\nab\vphi(0)\|_{L^2}^2 \right]
\exp\left\{\int_0^T a(s)\, ds\right\} 
	\\
&=0 \qquad \forall t\in [0,t_1].  \nonumber 
\end{align}
Here we have used the fact that $\int_0^T a(s)\, ds < \infty$.
Thus, $\vphi(t)=0$ for $t\in [0,t_1]$. Therefore, $\tau\geq t_1>0$. 
In fact, the above proof also shows that $\vphi(t)=0$ for $t\in [0,\tau]$.

Suppose that $\tau< T$, 
by the definition of $\tau$ we have $\vphi(\tau)=0$, 
that is, $\vphi_1(\tau)=\vphi_2(\tau)$. 
Repeating the above Gronwall's inequality argument with $\tau$ in place 
of $t=0$, we conclude that there exists $t_2>\tau$ such that
$\vphi(t)=0$ for $t\in [0,t_2]$. Hence, 
$\prodt{\vphi^2(s)}{\vphi_1(s)\vphi_2(s)}= 0$ $\forall s\in [0,t_2]$.
By the definition of $\tau$ we must have $\tau\geq t_2$.  
So we get a contradiction. Therefore, $\tau=T$ and $\vphi(t)=0$, 
\emph{i.e.}, $\vphi_1(t)=\vphi_2(t)$, for $t\in [0,T]$. 
The proof is complete. 
\end{proof}

	\section{Fully discrete finite element method}
	\label{sec-3}

	\subsection{Formulation of the finite element method}
	\label{sec-3.1}
For simplicity we assume that $\Omega\subset \mathbb{R}^d$  ($d=2$, 3) is a  polygonal or polyhedral domain.  Let $J_\tau=\{t_m\}_{m=0}^M$ be a uniform partition of $[0,T]$ of  mesh size $\tau:=\frac{T}{M}$, and $d_t v^m:=(v^m-v^{m-1})/\tau$.  (This is for simplicity; we could also use a quasi-uniform partition of the time interface.) Let $\Tc_h$ be a quasi-uniform ``triangulation" of the domain $\Ome$ of 
mesh size $h\in (0,1)$ and $\overline{\Omega} = \bigcup_{K \in \mathcal{T}_h} 
\overline{K}$ ($K\in \mathcal{T}_h$ are tetrahedrons in the case $d=3$). 
For a nonnegative integer $r$, let $P_r(K)$ denote the space of polynomials 
of degree less than or equal to $r$ on $K$, and define
	\begin{displaymath}
S_h^r = \left\{v_h\in C^0\big(\, \overline{\Ome}\, \big) \ \middle| \ v_h|_K\in P_r(K)\  
\forall K\in \Tc_h \right\} .
	\end{displaymath}
For fixed positive  integers $r$ and $\ell$, we introduce the finite element spaces $V_h = S_h^r$ and $W_h = S_h^\ell$. Define $\mathring{V}_h := V_h\cap L^2_0(\Ome)$ and similarly for $\mathring{W}_h$.

We now are ready to introduce our fully discrete finite element 
method for problem \eqref{eq2.1}, \eqref{eq2.2}, \eqref{eq1.8}--\eqref{eq1.11}
based on the variational formulation \eqref{eq2.3}--\eqref{eq2.5}. Find
$\bigl\{ (\phm, \mu_h^m, \vphihm \bigr)\}_{m=1}^M 
\subset \mathring{W}_h \times V_h\times V_h$ such that  
\begin{alignat}{2} \label{eq3.2}
\bigl(\nab\phm+\gamma \vphi_h^{m-1}\nab\mu_h^m,\nab q_h\bigr) &=0 
&&\quad \forall  q_h\in W_h, \\
\prodt{d_t\vphihm}{\nu_h} +\vep\bigl(\nab \mu_h^m,\nab\nu_h \bigr) 
\hskip 1.3in & && \label{eq3.3} \\
+\prodt{\vphi_h^{m-1}[\nab \phm+\gamma\vphi_h^{m-1}\nab \mu_h^m]}{\nab \nu_h} 
&= 0 &&\quad\forall \nu_h\in V_h, \nonumber \\
\prodt{\mu_h^m}{\psi_h}-\vep \prodt{\nab\vphihm}{\nab\psi_h}
-\frac{1}{\vep} \prodt{f_h^m}{\psi_h} &=0 &&\quad\forall \psi_h\in V_h, 
\label{eq3.4} \\
\vphi_h^0 &=\vphi_{0h}, && \label{eq3.5} 
\end{alignat}
where $\vphi_{0h}\in V_h$, to be specified in the next section, is an 
approximation of $\vphi_0^\vep$, and 
\begin{equation}
f_h^m:=  \left(\vphihm\right)^3 - \vphi_h^{m-1}.  \label{eq3.6}
\end{equation}
We will prove that $(W_h,V_h,V_h)$ is a stable triple for our mixed finite element approximation.  The techniques we use are based on energy estimates and convexity analysis, rather than an inf-sup-type condition, which would be used for the analysis of linear biharmonic-type equations~\cite{ciarlet,elliott89,XA2,scholz78}.

\subsection{Well-posedness of the finite element method} \label{sec-3.2}
The goal of this subsection is to show that the fully discrete 
finite element scheme \eqref{eq3.2}--\eqref{eq3.6}  is  
uniquely solvable and energy stable for all $h,\tau,\vep>0$.  
To prove unconditional unique solvability, we shall show that 
at each time step the scheme 
can be reformulated as a minimization problem for a strictly convex and 
coercive functional. We begin by defining an inner product on the 
subspace $\mathring{V}_h$.

\begin{lemma} \label{lem3.1}
Define the bilinear form $a:\mathring{V}_h\times\mathring{V}_h\to \mathbb{R}$ via
\begin{equation} \label{weak-form-L-h}
a(\mu,\nu):= \tau \left(\mathcal{M}\left(\vep,\vphi_h^{m-1}\right)\nabla\mu
+\vphi_h^{m-1}\nabla p(\mu),\nabla \nu \right)_{L^2},
\end{equation}
where $\mathcal{M}(\vep,\vphi) = \vep+\gamma\vphi^2$ and $p(\mu)\in 
\mathring{W}_h$ solves
\begin{equation} \label{pressure-mu-h-weak}
\left(\nabla p(\mu) ,\nabla q \right)_{L^2} 
=-\gamma \left(\vphi_h^{m-1}\nabla\mu,\nabla q \right)_{L^2} 
\qquad \forall q\in \mathring{W}_h \ .
\end{equation}
Then $a(\, \cdot \, , \, \cdot \, )$ is an inner product on $\mathring{V}_h$.  
\end{lemma}

The proof is omitted for brevity.  Note that in the next few calculations, 
the pressure, $p$, will be regarded as an auxiliary variable that can be 
calculated when the chemical potential, $\mu$, is known.  Owing to the last 
result, we can define 
an invertible linear operator $\mathcal{L} : \mathring{V}_h \rightarrow 
\mathring{V}_h$ via the following problem: given $\zeta\in \mathring{V}_h$, 
find $\mu\in \mathring{V}_h$ such that
\begin{equation}
a\left(\mu,\nu\right) = -\left(\zeta,\nu\right)_{L^2}  
\qquad \forall \nu\in V_h.
\end{equation}
This clearly has a unique solution because $a(\, \cdot \, , \, \cdot \, )$ 
is an inner product on $\mathring{V}_h$. We write $\mathcal{L}(\mu) = -\zeta$, 
or, equivalently, $\mu = -\mathcal{L}^{-1}(\zeta)$.

We now wish to define a negative norm, \emph{i.e.}, a discrete 
analogue to the $H^{-1}$ norm.  Again we omit the details for brevity.

\begin{lemma} \label{lem3.2}
Let $\zeta,\, \xi \in \mathring{V}_h$ and suppose 
$\mu_\zeta,\, \mu_\xi\in\mathring{V}_h$ are the unique weak 
solutions to ${\mathcal L}\left(\mu_\zeta\right) = -\zeta$ 
and ${\mathcal L}\left(\mu_\xi\right) = -\xi$.  Define
\begin{equation}
\left(\zeta,\xi\right)_{{\mathcal L}^{-1}} 
:=a\left(\mu_\zeta,\mu_\xi\right)
=-\left(\zeta,\mu_\xi\right)_{L^2} 
=-\left(\mu_\zeta,\xi\right)_{L^2}.
\label{crazy-inner-product-h}
\end{equation}
$\left(\, \cdot\, ,\, \cdot\, \right)_{{\mathcal L}^{-1}}$ defines an 
inner product on $\mathring{V}_h$, and the induced norm is 
\begin{equation}
\norm{\zeta}{\mathcal{L}^{-1}} 
= \sqrt{\left(\zeta,\zeta\right)_{{\mathcal L}^{-1}}}.  \label{crazy-norm-h}
\end{equation}
\end{lemma}

Using this last norm we can define a variational problem closely 
related to our fully discrete scheme. 

\begin{lemma}\label{lem3.3}
Set $K_1 := \left(\vphi_h^{m-1},1\right)_{L^2}$, and define 
$\vphi_\star^{m-1} :=\vphi_h^{m-1}-K_1 \in\mathring{V}_h$. 
For all $\vphi\in\mathring{V}_h$, define the nonlinear functional
\begin{equation}
G(\vphi) := \frac{1}{2}\norm{\vphi-\vphi_\star^{m-1}}{\mathcal{L}^{-1}}^2 
+\frac{1}{4\vep}\norm{\vphi+K_1}{L^4}^4 
+\frac{\vep}{2}\norm{\nabla\vphi}{L^2}^2 
-\frac{1}{\vep}\left(\vphi_h^{m-1},\vphi\right)_{L^2}.
\end{equation}
$G$ is strictly convex and coercive on the linear subspace $\mathring{V}_h$.  
Consequently, $G$ has a unique minimizer, call it 
$\vphi_\star^m\in\mathring{V}_h$.  Moreover, $\vphi_\star^m\in\mathring{V}_h$ 
is the unique minimizer of $G$ if and only if it is the unique solution to 
\begin{equation}
\frac{1}{\vep}\left(\left(\vphi_\star^m+K_1\right)^3,\psi\right)_{L^2}
+\vep\left(\nabla\vphi_\star^m,\nabla\psi\right)_{L^2} 
- \left(\mu_\star^m,\psi\right)_{L^2} 
= \frac{1}{\vep}\left(\vphi_h^{m-1},\psi\right)_{L^2} \label{nonlinear-1}
\end{equation}
for all $\psi\in\mathring{V}_h$, where $\mu_\star^m\in \mathring{V}_h$ 
is the unique solution to
\begin{equation}
a\left(\mu_\star^m,\nu\right) 
= -\left(\vphi_\star^m-\vphi_\star^{m-1},\nu\right)_{L^2} 
\qquad \forall  \nu\in\mathring{V}_h.  \label{nonlinear-2}
\end{equation}
\end{lemma}

\begin{proof}
In detail, the first variation, \emph{i.e.} the gradient, of the first 
term of $G$ is
\begin{eqnarray} 
\left.\frac{d}{ds}\left[\frac{1}{2}\norm{\vphi+s\psi-\vphi_\star^{m-1}}
{\mathcal{L}^{-1}}^2\right]\right|_{s=0} &=& \left.
\left[\left(\vphi+s\psi-\vphi_\star^{m-1},\psi\right)_{\mathcal{L}^{-1}}\right]
\right|_{s=0} \nonumber \\
&=& \left(\vphi-\vphi_\star^{m-1},\psi\right)_{\mathcal{L}^{-1}} \nonumber \\
&=& - \left(\mu,\psi\right)_{L^2}, \nonumber 
\end{eqnarray}
where, owing to the definition of the inner product 
$\left(\, \cdot \, , \, \cdot \, \right)_{\mathcal{L}^{-1}}$, 
$\mu\in \mathring{V}_h$ is the unique solution to 
\begin{equation}
a\left(\mu,\nu\right) = -\left(\vphi-\vphi_\star^{m-1},\nu\right)_{L^2} 
\qquad \forall \nu\in\mathring{V}_h.  
\end{equation}
The second variation is
\begin{equation*}
\left.\frac{d^2}{ds^2}\left[\frac{1}{2}\norm{\vphi+s\psi-\vphi_\star^{m-1}}
{\mathcal{L}^{-1}}^2\right]\right|_{s=0} 
=\left(\psi,\psi\right)_{\mathcal{L}^{-1}} > 0 \qquad \forall \psi \ne 0,
\end{equation*}
which establishes the strict convexity of the term. The strict 
convexity of $G$ follows because each of the other terms is at least convex.  
The coercivity of $G$ follows from an estimate of the form 
\begin{equation}
G(\vphi) \ge C_1(\vep) \norm{\vphi}{H^1}^2-C_2(\vep) \qquad \forall \vphi\in\mathring{V}_h, 	
\end{equation}
where $0<C_1(\vep),C_2(\vep)<\infty$ are constants. By the standard theory of convex optimization, 
$G$ has a unique (bounded) minimizer in $\mathring{V}_h$, call 
it $\vphi_\star^m$.  Moreover, $\vphi_\star^m$ is the unique minimizer 
of $G$ if and only if it is the unique solution to 
$A_{\vphi_\star^m}(\psi) = 0$, for all $\psi\in\mathring{V}_h$, where
\begin{eqnarray}
A_{\vphi}(\psi) := \frac{d}{ds}G(\vphi+s\psi)\Bigr|_{s=0} 
&=& \frac{1}{\vep}\bigl( (\vphi+K_1)^3,\psi \bigr) 
+ \vep \left(\nabla\vphi,\nabla\psi\right)_{L^2} \nonumber \\
&&+ \left(\vphi-\vphi_\star^{m-1},\psi\right)_{\mathcal{L}^{-1}} 
- \frac{1}{\vep}\left(\vphi_h^{m-1},\psi\right)_{L^2}.  
\end{eqnarray}
The rest of the details follow from the definition of the 
inner product $(\cdot,\cdot)_{\mathcal{L}^{-1}}$.
\end{proof}

Finally, we are in the position to prove the unconditional 
unique solvability of our scheme.

\begin{theorem} \label{thm-existence-uniqueness}
The scheme \eqref{eq3.2}--\eqref{eq3.4} is uniquely solvable for any 
mesh parameters $\tau$ and $h$ and for any phase parameter $\vep$.
\end{theorem}
	
\begin{proof}
First it is clear that a necessary condition for solvability of \eqref{eq3.3}
is that 
\begin{equation}
\left(\vphi_h^m,1\right)_{L^2} = \bigl(\vphi_h^{m-1},1\bigr)_{L^2} =: K_1,
\end{equation}
as can be found by taking $\nu_h\equiv 1$ in \eqref{eq3.3}.  
Now, let $\left(\mu_\star^m,\vphi_\star^m\right) 
\in\mathring{V}_h\times\mathring{V}_h$ be a solution 
of (\ref{nonlinear-1})--(\ref{nonlinear-2}). 
Define $\vphi^m_h := \vphi_\star^m+K_1/|\Omega|$.  Set
\begin{equation}
K_2:=\frac{1}{\vep}\Bigl(\bigl((\vphi_h^m)^3,1\bigr)_{L^2}- K_1\Bigr),
\end{equation}
and define $\mu_h^m:=\mu_\star^m+K_2/|\Omega|$.  Then it is straightforward to 
show that $\left(\mu_h^m,\vphi_h^m\right) \in V_h\times V_h$ is a solution to 
\eqref{eq3.3}--\eqref{eq3.4}.  In fact, there is a one-to-one 
correspondence of the respective solution sets.  Namely, 
if $\left(\mu_h^m,\vphi_h^m\right) \in V_h\times V_h$ is a solution 
to \eqref{eq3.2}--\eqref{eq3.4} then 
$\left(\mu_h^m-K_2/|\Omega|,\vphi_h^m-K_1/|\Omega|\right) \in\mathring{V}_h\times\mathring{V}_h$ 
is a solution to (\ref{nonlinear-1})--(\ref{nonlinear-2}).  
But (\ref{nonlinear-1})--(\ref{nonlinear-2}) admits only a unique solution, 
which proves that \eqref{eq3.2}--\eqref{eq3.4} is uniquely solvable.
\end{proof}

We now establish a discrete energy law for the numerical scheme
that mimics the continuous version \eqref{eq2.7}.

\begin{lemma} \label{lem4.1}
Let $(\phm,\mu_h^m,\vphihm)$ denote the unique solution of the scheme 
\eqref{eq3.3}--\eqref{eq3.6} and define 
$\buhm:=-\nab \phm-\gamma \vphi_h^{m-1}\nab\mu_h^m$, then there holds 
\begin{align} \label{eq4.1}
E_h^\ell &+\tau \sum_{m=1}^\ell\biggl\{  \vep\norm{\nab \mu_h^m}{L^2}^2 
+\frac{1}{\gamma}\norm{\buhm}{L^2}^2   
+ \frac{\tau}{4\vep} \Bigl[\,2\vep^2\norm{d_t\nab\vphihm}{L^2}^2 \\
&\quad +\norm{d_t(\vphihm)^2}{L^2}^2 + 2\norm{\vphihm d_t\vphihm}{L^2}^2 
+2\norm{d_t\vphihm}{L^2}^2 \, \Bigr] \biggr\} = E_h^0 
\nonumber
\end{align}
for all $0\leq \ell \leq M$. Here $E_h^m:=\Jc_\vep(\vphihm)$ 
and $\Jc_\vep(\, \cdot\, )$ is defined in \eqref{energy}.
\end{lemma}

\begin{proof}
The desired estimate \eqref{eq4.1} follows from setting $q_h=\phm$ 
in \eqref{eq3.2}, $\nu_h=\mu_h^m$ in \eqref{eq3.3}, 
$\psi_h=-d_t\vphihm$ in \eqref{eq3.4}, adding the resulting equations, 
using the identities 
\begin{align*}
\prodt{\nab\vphihm}{d_t \nab\vphihm} 
&=\frac12\,\bigl[\,d_t\norm{\nab \vphihm}{L^2}^2 
+\tau \norm{d_t \nab \vphihm}{L^2}^2 \,\bigr], \\
\prodt{f_h^m}{d_t\vphihm} &= \frac14\, d_t\norm{(\vphihm)^2-1}{L^2}^2 
+\frac{\tau}4 \bigl[\|d_t(\vphihm)^2\|_{L^2}^2 \\
&\hskip 1.4in 
+2\|\vphihm d_t\vphihm\|_{L^2}^2 +2\|d_t\vphihm\|_{L^2}^2 \, \bigr],
\end{align*}
and applying the operator $\tau\sum_{m=1}^\ell$ to the combined equation.
\end{proof}

The discrete energy law immediately implies the following uniform (in 
$\vep, h, \tau$) \emph{a priori} estimates for $(\phm,\mu_h^m,\vphihm)$.

\begin{lemma} \label{lem4.2}
Let $(\phm,\mu_h^m,\vphihm)$ be the unique solution of 
\eqref{eq3.2}--\eqref{eq3.6} and define 
$\buhm:=-\nab \phm-\gamma \vphi_h^{m-1}\nab\mu_h^m$. 
Suppose that $E_h^0<\infty$.  Then, for all $m\ge 1$,
	\begin{equation}
\int_\Ome\, \vphihm\, dx = \int_\Ome\, \vphi_h^0\, dx ,
	\label{eq4.2}
	\end{equation}
and, in addition, there hold the following estimates: 
	\begin{align}
\max_{0\leq m\leq M} \left[ \vep\norm{\nab \vphihm}{L^2}^2
+\frac{1}{\vep} \prodt{F(\vphihm)}{1} \right] &\leq C, 
	\label{eq4.3}
	\\
\max_{0\leq m\leq M}\norm{\vphi_h^m}{H^1}^2 &\le \frac{C}{\vep},
	\label{LinfH1phi-discrete}
	\\
\tau \sum_{m=1}^M \left[ \vep\norm{\nab\mu_h^m}{L^2}^2 
+\frac{1}{\gamma}\norm{\buhm}{L^2}^2  \right] &\leq C, 
	\label{eq4.4}
	\\
\sum_{m=1}^M \left[ \vep\norm{\nab\vphihm-\nab\vphi_h^{m-1}}{L^2}^2
+\frac{1}{\vep}\norm{\vphihm-\vphi_h^{m-1}}{L^2}^2 \right] &\leq C, 
	\label{eq4.5}
	\\
\sum_{m=1}^M \left[ \norm{\vphihm(\vphihm-\vphi_h^{m-1})}{L^2}^2
+\norm{(\vphihm)^2-(\vphi_h^{m-1})^2}{L^2}^2 \right] &\leq C\vep , 
	\label{eq4.6}
	\\
 \tau \sum_{m=1}^M \, \norm{\nab \phm}{L^{\frac32}}^2 &\leq \frac{C}{\vep^2}, 
	\label{eq4.7a} 
	\\
 \tau \sum_{m=1}^M \, \norm{d_t\vphihm}{(W^{1,3})^*}^2 &\leq \frac{C}{\vep} , 
	\label{eq4.7}
\end{align}
for some $\vep$, $h$, and $\tau$-independent constant $C=C(E_h^0)>0$.
\end{lemma}

\begin{proof}
\eqref{eq4.2} follows immediately from setting $\nu_h=1$ in \eqref{eq3.3},
and \eqref{eq4.3}--\eqref{eq4.6} are the immediate corollaries of the
discrete energy law \eqref{eq4.1}. \eqref{eq4.7a} follows from
the identity $\nab \phm= -\buhm-\gamma \vphi_h^{m-1}\nab\mu_h^m$,
the Sobolev embedding $H^1(\Ome)\hookrightarrow L^6(\Ome)$ for $d=2,\, 3$,
Young's inequality and estimates \eqref{LinfH1phi-discrete} and \eqref{eq4.4}.

Let $\mathcal{Q}_h$ denote the standard $L^2$ projection operator 
into $V_h$ (cf. \cite{bs08,ciarlet}), and for any $\nu\in W^{1,3}(\Ome)$, 
set $\nu_h=\mathcal{Q}_h\nu$ in \eqref{eq3.3}.  To prove \eqref{eq4.7}, 
we  use the Schwarz inequality and the Sobolev embedding 
$H^1(\Ome)\hookrightarrow L^6(\Ome)$ (for $d=2,3$) to we get
\begin{align*}
\prodt{d_t\vphihm}{\nu}&= \prodt{d_t\vphihm}{\mathcal{Q}_h\nu} 
	\\
&=-\vep\bigl(\nab \mu_h^m,\nab \mathcal{Q}_h\nu \bigr) 
+ \bigl(\vphi_h^{m-1}\buhm, \nab \mathcal{Q}_h\nu \bigr) 
	\\
&\leq \vep\norm{\nab\mu_h^m}{L^2} \norm{\nab \mathcal{Q}_h\nu}{L^2} 
+ \norm{\vphi_h^{m-1}}{L^6}\norm{\bu_h^m}{L^2} \norm{\nab \mathcal{Q}_h\nu}{L^3}\\
&\leq C\left[\vep\norm{\nab\mu_h^m}{L^2} 
+\norm{\vphi_h^{m-1}}{H^1} \norm{\bu_h^m}{L^2} \right]\, \norm{\nab \mathcal{Q}_h \nu}{L^3}\\
&\leq C\left[\vep\norm{\nab\mu_h^m}{L^2} 
+\frac{1}{\sqrt{\vep}} \norm{\bu_h^m}{L^2} \right]\, \norm{\nab \nu}{L^3},
\end{align*}
where we have used the $W^{1,3}$ stability of the $L^2$ projection
$\mathcal{Q}_h$ (cf. \cite{bs08,ciarlet}) to get the last inequality.
\eqref{eq4.7} now follows immediately from the above inequality and
estimates \eqref{eq4.3} and \eqref{eq4.4}. The proof is complete.
\end{proof}

\begin{remark}\label{rem4.1}
Property \eqref{eq4.2} says that the proposed numerical method enjoys
the same mass conservation law as the phase field model 
(cf. \eqref{eq2.6}).
This property will be validated numerically in Sec.~\ref{sec-5}.
\end{remark}

\section{Convergence analysis} \label{sec-4}

The goal of this section is to prove that the fully discrete finite 
element solution has a unique accumulation point (in some function space)
and this accumulation point is necessarily a weak solution 
to problem \eqref{eq2.3}--\eqref{eq2.5}. A byproduct of this
convergence result is to provide a constructive proof 
of the existence of weak solutions to problem \eqref{eq2.3}--\eqref{eq2.5}.

First, we derive some additional estimates for the finite
element solution. To the end, we introduce the discrete 
Laplacian $\Del_h: V_h\to V_h$ which is defined as follows: 
for any $v_h\in V_h$, $\Del_h v_h\in V_h$ denotes the unique solution 
to the problem
	\begin{equation}
\bigl(\Del_h v_h,  w_h\bigr) =-\bigl(\nab v_h,\nab w_h\bigr) 
\qquad\forall w_h\in V_h.
	\label{eq4.8}
	\end{equation}
In particular, setting $w_h = \Del_h v_h$ in \eqref{eq4.8}, we obtain
\begin{displaymath}
\norm{\Del_h v_h}{L^2}^2 = -\left(\nab v_h,\nab \Del_h v_h\right).
\end{displaymath}

\begin{lemma} \label{lem4.3}
Let $(\phm,\mu_h^m,\vphihm)$ and $\buhm$ be same as in 
Lemma \ref{lem4.2}. Then, under the assumption $E_h^0<\infty$, there 
hold the following additional estimates: for $d=2,3$, 
	\begin{align}
\tau \sum_{m=1}^M \norm{\Del_h \vphihm}{L^2}^2 &\leq C_0(T+1) C(\vep),
	\label{eq4.10} 
	\\
\tau \sum_{m=1}^M \norm{\vphihm}{L^\infty}^{\frac{4(6-d)}{d}} 
&\leq C_0(T+1)C(\vep), 
	\label{eq4.10a} 
	\\
\tau \sum_{m=1}^M \norm{\nabla \vphihm}{L^4}^{\frac{8}d}
&\leq C_0(T+1)C(\vep), 
	\label{eq4.11} 
	\\
\tau \sum_{m=1}^M \norm{\nab\phm}{L^2}^{\frac{4(6-d)}{12-d}}
&\leq C_0(T+1)C(\vep), 
	\label{eq4.12}
	\\
\tau \sum_{m=1}^M \norm{d_t\vphihm}{(H^1)^*}^{\frac{4(6-d)}{12-d}} 
&\leq C_0(T+1)C(\vep) ,
	\label{eq4.13}
	\end{align}
for some $\vep,h$, and $\tau$-independent constant $C_0=C_0(E_h^0)>0$ 
and some $h$ and $\tau$-independent constant $C(\vep)>0$ that grows like $\vep^{-r}$, for
some $r\in\mathbb{Z}^+$, as $\vep\to 0$.
\end{lemma}

\begin{proof}
Setting $\psi_h=\Del_h \vphihm$ in \eqref{eq3.4}, using the 
definition of $\Del_h\vphihm$, and the Schwarz inequality we get
\begin{align*}
\vep \norm{\Del_h \vphihm}{L^2}^2 
&= - \vep \bigl(\nab\vphihm,\nab\Del_h \vphihm\bigr) \\ 
&=-\bigl(\mu_h^m,\Del_h \vphihm\bigr)
+ \frac{1}{\vep}\bigl(f_h^m, \Del_h \vphihm \bigr) \\
&\le\bigl(\nab\mu_h^m,\nab\vphihm\bigr)
+ \frac{1}{\vep}\left(\frac{\vep^2}{2} \norm{\Del_h \vphihm}{L^2}^2 
+ \frac{1}{2\vep^2} \norm{f_h^m}{L^2}^2 \right) \\
&\leq \frac12\norm{\nab\mu_h^m}{L^2}^2 +\frac12 \norm{\nab\vphihm}{L^2}^2  
+\frac{\vep}2 \norm{\Del_h \vphihm}{L^2}^2 
+\frac{1}{2\vep^3} \norm{f_h^m}{L^2}^2.
\end{align*}
Hence,
\begin{align}\label{eq4.13a}
\vep \norm{\Del_h \vphihm}{L^2}^2 \leq \norm{\nab\mu_h^m}{L^2}^2  
+\norm{\nab\vphihm}{L^2}^2 + \frac{1}{\vep^3}\norm{f_h^m}{L^2}^2.
\end{align}

To bound $\norm{f_h^m}{L^2}^2$, we write
\[
f_h^m:=(\vphihm)^3-\vphi_h^{m-1} = \vphihm \bigl((\vphihm)^2-1\bigr)
+ \vphihm-\vphi_h^{m-1}.
\]
Then by \eqref{eq4.3} we have
\begin{align*} 
\|f_h^m\|_{L^2}^2 &\leq 2\|\vphihm\|_{L^\infty}^2 \|(\vphihm)^2-1\|_{L^2}^2
+\|\vphihm-\vphi_h^{m-1}\|_{L^2}^2 \\
&= 8\|\vphihm\|_{L^\infty}^2 \left(F(\vphihm),1\right)  
+\|\vphihm-\vphi_h^{m-1}\|_{L^2}^2  \\
&\leq C\vep \|\vphihm\|_{L^\infty}^2 +  \|\vphihm-\vphi_h^{m-1}\|_{L^2}^2.
\end{align*}
We now appeal to the following discrete Gagliardo-Nirenberg
inequality (cf. \cite{hr1} and \cite{adams,elliott86}):
	\begin{align}
	\label{eq4.13b}
\|\vphihm\|_{L^\infty} \leq C\|\Delta_h \vphi_h^m\|_{L^2}^{\frac{d}{2(6-d)}}
\,\|\vphihm\|_{L^6}^{\frac{3(4-d)}{2(6-d)}} + C\|\vphihm\|_{L^6} \qquad (d=2,3)
	\end{align}
and get
	\begin{align}
	\label{eq4.13c}
\|f_h^m\|_{L^2}^2 &\leq C\vep \|\Delta_h \vphi_h^m\|_{L^2}^{\frac{d}{6-d}}\, 
\|\vphihm\|_{L^6}^{\frac{3(4-d)}{6-d}}+ C\vep\|\vphihm\|_{L^6}^2
+ \|\vphihm-\vphi_h^{m-1}\|_{L^2}^2 
	\\
&\leq \vep\left( \frac{\vep^3}{2}\|\Delta_h \vphi_h^m\|_{L^2}^2 +C\vep^{\frac{-d}{4-d}} \norm{\vphihm}{L^6}^2 \right) + C\vep\norm{\vphihm}{L^6}^2
+ \norm{\vphihm-\vphi_h^{m-1}}{L^2}^2
	\nonumber
	\\	
&\leq \frac{\vep^4}2 \|\Delta_h \vphi_h^m\|_{L^2}^2
+ C\left(\vep^{\frac{4-2d}{4-d}}+\vep\right)\norm{\vphihm}{L^6}^2 
+ \norm{\vphihm-\vphi_h^{m-1}}{L^2}^2
	\nonumber  
	\\	
&\leq \frac{\vep^4}2 \|\Delta_h \vphi_h^m\|_{L^2}^2
+ C\left(\vep^{\frac{4-2d}{4-d}}+\vep\right)\norm{\vphihm}{H^1}^2 
+ \norm{\vphihm-\vphi_h^{m-1}}{L^2}^2,
	\nonumber  
	\end{align}
where we used the Sobolev embedding $H^1(\Ome)\hookrightarrow L^6(\Ome)$ 
for $d=2$, $3$ in the last step.  Then \eqref{eq4.10} follows from applying 
the operator $\tau\sum_{m=1}^M$ to \eqref{eq4.13a} and using \eqref
{eq4.13c}, \eqref{LinfH1phi-discrete}, and \eqref{eq4.4}.
\eqref{eq4.10a} is an immediate consequence of \eqref{eq4.13b} and 
\eqref{eq4.10}.

To prove \eqref{eq4.11}, we recall another discrete Gagliardo-Nirenberg
inequality (cf. \cite{hr1} and \cite{adams,elliott86}):
\begin{align} \label{eq4.13d}
\norm{\nabla \nu_h}{L^4}  
\leq C \norm{\nabla \nu_h}{L^2}^{\frac{4-d}{4}} 
\norm{\Delta_h \nu_h}{L^2}^{\frac{d}{4}}
+C\norm{\nabla \nu_h}{L^2} \qquad \forall \nu_h \in V_h,\quad d=2,\, 3.
\end{align}
It follows from the above inequality and estimates \eqref{eq4.3} 
and \eqref{eq4.10} that
	\begin{eqnarray}
\tau\sum_{m=1}^M \norm{\nabla \vphihm}{L^4}^{\frac{8}{d} } &\leq& C\, \tau\sum_{m=1}^M\|\nab\vphihm\|_{L^2}^{\frac{2(4-d)}{d}}\, \|\Delta_h\vphihm\|_{L^2}^2 + C\, \tau\sum_{m=1}^M \norm{\nabla \vphihm}{L^2}^{\frac{8}{d}}
	\nonumber
	\\
&\leq& \frac{C}{\vep^{\frac{4-d}{d}}}\, \tau\sum_{m=1}^M  \|\Delta_h\vphihm\|_{L^2}^2 + \frac{C}{\vep^{\frac{4}{d}}}\, T 
	\end{eqnarray}
which proves \eqref{eq4.11}.

Inequality \eqref{eq4.12} follows from the estimate
	\begin{eqnarray}
\norm{\nab \phm}{L^2}^{\frac{4(6-d)}{12-d}} &\le& C\norm{\buhm}{L^2}^{\frac{4(6-d)}{12-d}}+ C \norm{\vphi_h^{m-1}}{L^\infty}^{\frac{4(6-d)}{12-d}}\norm{\nab\mu_h^m}{L^2}^{\frac{4(6-d)}{12-d}}
	\nonumber
	\\
&\le& C\norm{\buhm}{L^2}^2+ C \norm{\vphi_h^{m-1}}{L^\infty}^{\frac{4(6-d)}{d}} + C\norm{\nab\mu_h^m}{L^2}^2
	\nonumber
	\end{eqnarray}
and estimates \eqref{eq4.4} and \eqref{eq4.10a}.

Now, let $\mathcal{Q}_h$ denote the standard $L^2$ projection 
operator into $V_h$ (cf. \cite{bs08,ciarlet}).  For any $\nu\in H^1(\Ome)$, 
setting $\nu_h=\mathcal{Q}_h\nu$ in \eqref{eq3.3}, we get
\begin{align*}
\prodt{d_t\vphihm}{\nu}&= \prodt{d_t\vphihm}{\mathcal{Q}_h\nu} 
	\\
&=-\vep\bigl(\nab \mu_h^m,\nab \mathcal{Q}_h\nu \bigr)  
+ \bigl(\vphi_h^{m-1}\buhm, \nab \mathcal{Q}_h\nu \bigr) 
	\\
&\leq \Bigl[\vep\|\nab\mu_h^m\|_{L^2} 
+ \|\vphi_h^{m-1}\|_{L^\infty} \|\bu_h^m\|_{L^2} \Bigr]\,
\|\nab \mathcal{Q}_h\nu\|_{L^2} 
	\\
&\leq C\Bigl[\vep\|\nab\mu_h^m\|_{L^2} 
+\|\vphi_h^{m-1}\|_{L^\infty} \|\bu_h^m\|_{L^2} \Bigr]\, \|\nab \nu\|_{L^2},
\end{align*}
where we have used the $H^1$ stability of the $L^2$ projection 
$\mathcal{Q}_h$ (cf. \cite{bs08,ciarlet}) to get the last inequality.
\eqref{eq4.13} now follows immediately from the above inequality and
estimates  \eqref{eq4.10a} and  \eqref{eq4.4}. 
The proof is complete.
\end{proof}


\medskip
Next, let $\vphi_{h,\tau}(x,t)$ denote the piecewise linear interpolant 
(in $t$) of the fully discrete solution $\{\vphihm\}$, that is,
\begin{equation}
\vphi_{h,\tau}(\, \cdot\, ,t) := \frac{t-t_{m-1}}{\tau}\vphihm(\, \cdot\, ) 
+\frac{t_m-t}{\tau} \vphi_h^{m-1}(\, \cdot\, )
\qquad\forall t\in[t_{m-1}, t_m], 
\label{eq4.17}
\end{equation}
for $1\leq m\leq M$. Let $\op_{h,\tau}(x,t)$, $\obu_{h,\tau}(x,t)$, 
 $\omu_{h,\tau}(x,t)$, $\ovphi_{h,\tau}(x,t)$, and $\oovphi_{h,\tau}(x,t)$, 
denote the piecewise constant extensions of $\{\phm\}$, $\{\buhm\}$,  
$\{\mu_h^m\}$, and $\{\vphihm\}$, respectively, defined as follows
\begin{alignat}{2}
\op_{h,\tau}(\cdot,t)&:=\phm &&\qquad \forall  t\in[t_{m-1}, t_m],\quad
1\leq m\leq M, \label{eq4.18} \\
\obu_{h,\tau}(\cdot,t)&:=\buhm &&\qquad \forall  t\in[t_{m-1}, t_m],\quad
1\leq m\leq M, \label{eq4.19} \\
\omu_{h,\tau}(\cdot,t)&:=\mu_h^m &&\qquad \forall  t\in[t_{m-1},
t_m],\quad 1\leq m\leq M,  \label{eq4.22} \\
\ovphi_{h,\tau}(\cdot,t)&:=\vphihm &&\qquad \forall t\in[t_{m-1},
t_m],\quad 1\leq m\leq M, \label{eq4.20} \\
\oovphi_{h,\tau}(\cdot,t)&:=\vphi_h^{m-1} &&\qquad \forall  t\in[t_{m-1},
t_m],\quad 1\leq m\leq M. \label{eq4.21}
\end{alignat}
We remark that $\vphi_{h,\tau}(x,t)$ is a continuous piecewise polynomial 
function in space and time, $\op_{h,\tau}(x,t)$, $\obu_{h,\tau}(x,t)$, 
$\omu_{h,\tau}(x,t)$, and  $\ovphi_{h,\tau}(x,t)$ are right continuous at 
the nodes $\{t_m\}$, and $\oovphi_{h,\tau}(x,t)$ is left continuous at 
the nodes $\{t_m\}$.

The main result of this section is the following convergence theorem.

\begin{theorem} \label{thm4.1}
Let $\Omega\subset\mathbb{R}^d$ ($d=2, 3$) be a
polygonal or polyhedral domain. For each fixed $\vep>0$, suppose that 
$\Jc_\vep(\vphi_{0h})\le C_0<\infty$, where $C_0$ is independent of $h$, and 
\begin{displaymath}
\lim_{h\to 0} \norm{\vphi_{0h}-\vphi_0^\vep}{L^2}=0.
\end{displaymath}
Then the sequence $\{(\op_{h,\tau},\obu_{h,\tau},\omu_{h,\tau},
\ovphi_{h,\tau})\}$ has an accumulation point 
$(p^\vep, \bu^\vep,\mu^\vep,\vphi^\vep)$ with
$\bu^\vep=-\nab p^\vep-\gamma \vphi^\vep\nab \mu^\vep$,
and $(p^\vep,\mu^\vep,\vphi^\vep)$ is a weak solution 
to problem \eqref{eq2.3}--\eqref{eq2.5}.
\end{theorem}

\begin{proof} 
We divide the proof into two steps.

{\em Step 1: Extracting convergent subsequences.} The estimates of 
Lemmas \ref{lem4.2} and \ref{lem4.3} immediately give the following 
(uniform in $h$ and $\tau$) estimates:
	\begin{eqnarray}
 \norm{\nab \ovphi_{h,\tau} }{L^\infty(L^2)} +\norm{\ovphi_{h,\tau}^2-1 }{L^\infty(L^2)} &\leq& C,
	\label{eq4.23}
	\\
\norm{\obu_{h,\tau} }{L^2(L^2)}+\norm{\nab \omu_{h,\tau}}{L^2\left(L^2\right)}  
&\leq& C, 
	\label{eq4.24}
	\\
\norm{\nab \op_{h,\tau}}{L^2\left(L^{\frac32}\right)} + \norm{\nab \op_{h,\tau}}{L^\sigma(L^2)} &\leq& C, 
	\label{eq4.25}
	\\
\norm{\ovphi_{h,\tau}}{L^\beta\left(L^\infty\right)} 
&\leq& C, 
	\label{eq4.25a}
	\\
\norm{\nab\ovphi_{h,\tau}}{L^{\frac{8}d}\left(L^4\right)}
&\leq& C,
	\label{eq4.25b}
	\\ 
\norm{(\vphi_{h,\tau})_t}{L^2\left(\left(W^{1,3}\right)^*\right)} +\norm{(\vphi_{h,\tau})_t }{L^\sigma\left(\left(H^1\right)^*\right)} 
&\leq& C, 
	\label{eq4.26}
	\end{eqnarray}
where $\beta := {\frac{4(6-d)}{d}}\ge 4$ and $\sigma :=\frac{4(6-d)}{12-d} \ge \frac{4}{3}$.  Note, we have suppressed the dependences of the constants on $T$ and $\vep$ above. $\{\nab  \op_{h,\tau}\}$ is uniformly (with respect to $h$ and $\tau$)
integrable in $L^\sigma\left((0,T); L^2(\Ome)\right)$ and 
$\{(\vphi_{h,\tau})_t\}$ is uniformly integrable in 
$L^\sigma\left((0,T);(H^1)^*\right)$. 

Then there exists a convergent subsequence of 
$\{(\op_{h,\tau},\obu_{h,\tau},\omu_{h,\tau},\vphi_{h,\tau})\}$ 
(still denoted by the same symbols) and a quadruple 
$(p^\vep,\bu^\vep,\mu^\vep,\vphi^\vep)$ such that
\begin{alignat*}{2}
&p^\vep\in L^\sigma\left((0,T); H^1(\Ome)\cap L^2_0(\Ome)\right),
\qquad &&\bu^\vep\in L^2\left((0,T);\bL^2(\Ome)\right), \\
&\vphi^\vep\in L^\infty\left((0,T);H^1(\Ome)\right)
\cap L^{4d}\left((0,T);L^\infty(\Ome)\right),
\qquad && \vphi^\vep\in L^{\frac{8}d}\left((0,T);W^{1,4}(\Ome)\right), \\
&\vphi^\vep_t \in L^\sigma\left((0,T);(H^1(\Ome))^*\right)
\cap L^2\left((0,T);(W^{1,3}(\Ome))^*\right), \qquad 
&& \mu^\vep\in L^2\left((0,T);H^1(\Ome)\right),
\end{alignat*}
and
\begin{align}
\op_{h,\tau} \overset{h,\tau\searrow 0}\longrightarrow p^\vep 
&\,\mbox{weakly in } L^\sigma\left((0,T); H^1(\Ome)\cap L^2_0(\Ome)\right),
\label{eq4.29} \\
\obu_{h,\tau}\overset{h,\tau\searrow 0}{\longrightarrow} \bu^\vep 
&\,\mbox{weakly in } L^2\left((0,T);\bL^2(\Ome)\right), \label{eq4.27} \\
\omu_{h,\tau} \overset{h,\tau\searrow 0}\longrightarrow
\mu^\vep &\,\mbox{weakly in } L^2\left((0,T);H^1(\Ome)\right),  \label{eq4.31} \\
\ovphi_{h,\tau}\overset{h,\tau\searrow 0}\longrightarrow \vphi^\vep 
&\,\mbox{weakly$\star$ in } L^\infty\left((0,T);H^1(\Ome)\right)
\cap L^\beta\left((0,T);L^\infty(\Ome)\right),
\label{eq4.28} \\
&\,\mbox{strongly in } L^2\left((0,T);L^2(\Ome)\right), \nonumber \\
&\,\mbox{weakly in } H^1\left((0,T);(W^{1,3}(\Ome))^*\right)
\cap W^{1,\sigma}\left((0,T);(H^1(\Ome))^*\right), \nonumber \\
&\,\mbox{weakly in } L^{\frac{8}d}\left((0,T);W^{1,4}(\Ome)\right). \nonumber
\end{align}
We have used Aubin-Lions lemma (cf. \cite{simon85}) to conclude \eqref{eq4.28}.

From \eqref{eq4.5} we also have
\begin{eqnarray*}
\norm{\vphi_{h,\tau}-\ovphi_{h,\tau}}{L^2(H^1)}^2
&=&\sum_{m=1}^M \norm{\vphihm-\vphi_h^{m-1}}{H^1}^2 
\int_{t_{m-1}}^{t_m} \Bigl(\frac{t-t_{m-1}}{\tau}\Bigr)^2 dt \\
&=&\frac{\tau}3 \sum_{m=1}^M \norm{\vphihm-\vphi_h^{m-1}}{H^1}^2
\overset{\tau\searrow 0}{\longrightarrow} 0.
\end{eqnarray*}
Hence, $\{\vphi_{h,\tau}\}$, $\{\ovphi_{h,\tau}\}$,
and $\{\oovphi_{h,\tau}\}$ converge to the same limit as $h,\tau\rightarrow 0$. 

{\em Step 2: Passing to the limit.} We now want to pass to the 
limit in \eqref{eq3.2}--\eqref{eq3.5} and to show that 
$(p^\vep,\mu^\vep,\vphi^\vep)$ is a weak solution to problem 
\eqref{eq2.3}--\eqref{eq2.5} with the initial data
$\vphi^\vep(0)=\vphi_0^\vep$.  To this end, we rewrite 
\eqref{eq3.2}--\eqref{eq3.5} as
\begin{alignat}{2}
\bigl(\obu_{h,\tau},\nab q_h\bigr) &=0 &&\qquad\forall q_h\in W_h,
	\label{eq4.35}
	\\
\prodt{(\vphi_{h,\tau})_t}{\nu_h} +\vep\bigl(\nab\omu_{h,\tau},\nab\nu_h\bigr) 
-\bigl(\oovphi_{h,\tau}\obu_{h,\tau},\nab\nu_h\bigr) &=0
&&\qquad\forall\nu_h\in V_h, 
	\label{eq4.36}
	\\
\prodt{\omu_{h,\tau}}{\psi_h} -\vep\prodt{\nab \ovphi_{h,\tau}}{\nab \psi_h}
-\frac{1}{\vep}\prodt{\of_{h,\tau}}{\psi_h} &=0 &&\qquad\forall\psi_h\in V_h,
	\label{eq4.37}
\end{alignat}
where $\of_{\vep,h,\tau}$ denotes the right continuous
constant extension of $\{f_h^m\}$. 

For any $\eta\in C^0([0,T])$, multiplying \eqref{eq4.35}--\eqref{eq4.37}
by $\eta$, respectively, and integrating the resulting equations
in $t$ from $0$ to $T$ we get
\begin{align}
\int_0^T \bigl(\obu_{h,\tau}, \nab q_h\bigr)\eta(t)\, dt =0
& \quad\forall q_h\in W_h, 
	\label{eq4.38}
	\\
\int_0^T \Bigl\{ \prodt{(\vphi_{h,\tau})_t}{\nu_h} 
+\vep\bigl(\nab\omu_{h,\tau},\nab\nu_h\bigr)  \hskip 1.0in &  \label{eq4.39}\\
-\bigl(\oovphi_{h,\tau}\obu_{h,\tau},\nab\nu_h\bigr)\Bigr\}\eta(t)\, dt =0
&\quad\forall \nu_h\in V_h, 
	\nonumber
	\\
\int_0^T \left\{ \prodt{\omu_{h,\tau}}{\psi_h}
-\vep\prodt{\nab \ovphi_{h,\tau}}{\nab \psi_h}
-\frac{1}{\vep}\prodt{\of_{h,\tau}}{\psi_h} \right\}\eta(t)\, dt
=0  &\quad\forall \psi_h\in V_h.  
	\label{eq4.40}
\end{align}

For any $(q,\nu,\psi)\in [H^1(\Ome)\cap C^1(\Ome)]^3$, let $(q_h,\nu_h,\psi_h) 
\in W_h\times V_h\times V_h$ be the standard finite  element (nodal)
interpolations of $(q,\nu,\psi)$ in  \eqref{eq4.38}--\eqref{eq4.40}. Since
\begin{displaymath}
q_h\overset{h\searrow 0}{\longrightarrow} q, \quad \nu_h\overset{h\searrow 0}{\longrightarrow} \nu, \quad \psi_h\overset{h\searrow 0}{\longrightarrow} \psi \quad\mbox{strongly in } H^1(\Ome), 
\end{displaymath}
sending $h,\tau\rightarrow 0$ in \eqref{eq4.38}--\eqref{eq4.40} 
and using \eqref{eq4.27}--\eqref{eq4.31} we get $\vphi^\vep(0)=\vphi^\vep_h$ and
\begin{alignat}{2}
\int_0^T \bigl(\bu^\vep,\nab q\bigr)\eta(t)\, dt &=0 
&&\quad\forall q\in H^1(\Ome), 
	\label{eq4.41}
	\\
\int_0^T \Bigl\{ \dual{\vphi^\vep_t}{\nu} +\vep\bigl(\nab\mu^\vep,\nab\nu\bigr) 
+\bigl(\vphi^\vep\bu^\vep,\nab\nu\bigr)\Bigr\}\eta(t)\, dt &=0 
&&\quad\forall \nu\in H^1(\Ome), 
	\label{eq4.42}
	\\
\int_0^T \Bigl\{ \prodt{\mu^\vep}{\psi} -\vep\prodt{\nab \vphi^\vep}{\nab \psi}
-\frac{1}{\vep}\prodt{f(\vphi^\vep)}{\psi} \Bigr\}\eta(t)\, dt
&=0  &&\quad\forall \psi\in H^1(\Ome).  
	\label{eq4.43}
\end{alignat}

Moreover, from the identity $\buhm=-\nab \phm-\gamma \vphi_h^{m-1}\nab\mu_h^m$ 
we have 
\begin{align*}
\int_0^T \bigl(\obu_{h,\tau}, \nab q\bigr) \eta(t)\, dt
&=-\int_0^T\bigl(\nab \op_{h,\tau}
+\gamma \oovphi_{h,\tau}\nab\omu_{h,\tau},\nab q \bigr) \eta(t)\, dt, \\
\int_0^T \bigl(\ovphi_{h,\tau}\obu_{h,\tau}, \nab \nu\bigr) \eta(t)\, dt
&=-\int_0^T \bigl(\ovphi_{h,\tau}[\nab \op_{h,\tau}
+\gamma \oovphi_{h,\tau}\nab\omu_{h,\tau}],\nab \nu\bigr) \eta(t)\, dt.
\end{align*}
Sending $h,\tau\rightarrow 0$ and using \eqref{eq4.29}--\eqref{eq4.31} yields
\begin{align}
\int_0^T \bigl(\bu^\vep, \nab q\bigr) \eta(t)\, dt
&=-\int_0^T\bigl(\nab p^\vep+\gamma \vphi^\vep\nab\mu^\vep,\nab q\bigr)\eta(t)\,dt,
\label{eq4.44} \\
\int_0^T \bigl(\vphi^\vep\bu^\vep, \nab \nu\bigr) \eta(t)\, dt
&=-\int_0^T \bigl(\vphi^\vep[\nab p^\vep+\gamma \vphi^\vep\nab\mu^\vep],
\nab \nu\bigr) \eta(t)\, dt.  \label{eq4.45}
\end{align}

Combining \eqref{eq4.41}--\eqref{eq4.43} and \eqref{eq4.44}--\eqref{eq4.45} 
we obtain \eqref{eq2.3}--\eqref{eq2.5}, since $C^0[0,T]$ is dense 
in $L^2(0,T)$. Hence, $(p^\vep,\mu^\vep,\vphi^\vep)$ is a weak solution 
to \eqref{eq2.3}--\eqref{eq2.5}.  The proof is complete.
\end{proof}

\begin{corollary} \label{cor-4.1}
The whole sequence 
$\{(\op_{h,\tau},\obu_{h,\tau},\omu_{h,\tau},\ovphi_{h,\tau})\}$ 
converges if weak solutions to problem \eqref{eq2.3}--\eqref{eq2.5} are unique.
\end{corollary}

\begin{proof}
We have shown in the above proof that
$\{(\op_{h,\tau},\omu_{h,\tau},\ovphi_{h,\tau})\}$ has a 
convergent subsequence and its limit $(p^\vep,\mu^\vep,\vphi^\vep)$ 
is a weak solution of \eqref{eq2.3}--\eqref{eq2.5}. 
Moreover, the proof also implies that the limit of
{\em every} convergent subsequence of 
$\{(\op_{h,\tau},\omu_{h,\tau},\ovphi_{h,\tau})\}$ is necessarily a weak 
solution of \eqref{eq2.3}--\eqref{eq2.5}. Hence, by the uniqueness
assumption of weak solutions, the whole sequence 
$\{(\op_{h,\tau},\omu_{h,\tau},\ovphi_{h,\tau})\}$ must
converge to the unique weak solution.
\end{proof}

\begin{remark}\label{rem4.3}
In Theorem \ref{thm4.1} and Corollary \ref{cor-4.1}, $\Omega$ is 
assumed to be a polygonal or polyhedral domain. This assumption is imposed 
only to avoid the technicalities for defining our finite element 
method \eqref{eq3.2}--\eqref{eq3.6}. It is not used or needed in the 
proofs of the theorem and the corollary. By using the standard numerical
integration technique or the approximated boundary technique 
(\emph{i.e.}, to approximate a bounded Lipschitz domain by a sequence of 
polygonal or polyhedral domains) (cf. \cite{ciarlet}), it can be 
proved that the modified finite element methods would also 
possess all the properties proved in Lemmas \ref{lem3.1}, \ref{lem3.2},
\ref{lem3.3}, \ref{lem4.1}, \ref{lem4.2}, and 
Theorem \ref{thm-existence-uniqueness} as well as Lemma \ref{lem4.3}. 
As a result, the conclusions of Theorem \ref{thm4.1} and 
Corollary \ref{cor-4.1} still hold when $\Ome$ is a bounded 
Lipschitz domain.
\end{remark}

From Theorem \ref{thm4.1} and Theorem  \ref{thm2.1} we immediately have

\begin{theorem}\label{thm4.2}
There exists a weak solution to problem \eqref{eq2.3}--\eqref{eq2.5} and
weak solutions are unique in the function class $\mathcal{F}$.
\end{theorem}

We conclude this section with a remark on the error estimates
for the solution of the fully discrete scheme (\ref{eq3.2})--(\ref{eq3.5}). 
Using the standard (perturbation) technique as presented 
in \cite{feng07}, it is not hard to prove that the scheme converges
optimally in the energy norm. However, the error constant 
would contain a factor of $\exp\left(\vep^{-2}\right)$. Such an 
error bound is clearly not very useful for small $\vep$. 
A better error bound would only depend on $\vep^{-1}$ in
some low polynomial order (cf. \cite{XA2,XA3}). Deriving such
a polynomial order error bound is an on-going project and the
result will be reported in a forthcoming paper.

\section{Numerical experiments} \label{sec-5}

In this section we provide some numerical experiments to gauge the accuracy
and reliability of the fully discrete finite element method developed in
the previous sections.  For the experiments we take $V_h=W_h =S_h^1$ for simplicity. We use a square domain $\Omega=(0,1)^2$ and take $\Tc_h$ to be a regular
triangulation of $\Omega$ consisting of right isosceles triangles, as
depicted in Fig.~\ref{fig3}.  We use a nonlinear multigrid method, which is
detailed in Appendix~\ref{app-multigrid-solver}, to solve the scheme
(\ref{eq3.2})--(\ref{eq3.5}) at each time step. We perform a battery of three tests
on the scheme.  First, we measure numerical convergence of the scheme in
the presence of added, artificial source terms.  Second, we measure the
numerical convergence of the scheme without source terms using a
Cauchy-convergence method.  Third, we conduct a test of spinodal
decomposition using varying values of the excess surface tension $\gamma$,
and demonstrate the discrete energy dissipation and mass conservation properties of the scheme.

	\begin{table}[t]
	\centering
	\begin{tabular}{cccccccc}
$h$ & $\norm{e_\vphi}{L^2}$ & rate & $\norm{e_\mu}{L^2}$ & rate & $\norm{e_p}{L^2}$ & rate
	\\
	\hline
$\nicefrac{\sqrt{2}}{16}$ & $8.683\times 10^{-3}$ & -- & $1.088\times 10^{-2}$ & -- & $1.270\times 10^{-2}$ & --
	\\
$\nicefrac{\sqrt{2}}{32}$ & $1.850\times 10^{-3}$ & 2.23 & $2.701\times 10^{-3}$ & 2.01 & $2.479\times 10^{-3}$ & 2.35
	\\
$\nicefrac{\sqrt{2}}{64}$ & $4.568\times 10^{-4}$ & 2.01 & $6.759\times 10^{-4}$ & 2.00 & $5.759\times 10^{-4}$ & 2.11
	\\
$\nicefrac{\sqrt{2}}{128}$ & $1.141\times 10^{-4}$ & 2.00 & $1.691\times 10^{-4}$ & 2.00 & $1.413\times 10^{-4}$ & 2.03
	\\
$\nicefrac{\sqrt{2}}{256}$ & $2.852\times 10^{-5}$ & 2.00 & $4.227\times 10^{-5}$ & 2.00 & $3.515\times 10^{-5}$ &  2.00
	\\
	\hline
	\end{tabular}
\caption{$L^2$ convergence test. The final time is $T = 1.0$, and the 
refinement path is taken to be $\tau = 25.6h^2$. The other parameters 
are $\varepsilon = \gamma = 1.0$; $\Omega = (0,1)^2$. The global error 
at $T$ is expected to be $\mathcal{O}(\tau)+\mathcal{O}\left(h^2\right) 
= \mathcal{O}\left(h^2\right)$, and this is confirmed.}
	\label{tab1}
	\end{table}

	\begin{table}
	\centering
	\begin{tabular}{cccccccc}
$h$ & $\norm{e_\vphi}{H^1}$ & rate & $\norm{e_\mu}{H^1}$ & rate & $\norm{e_p}{H^1}$ & rate
	\\
	\hline
$\nicefrac{\sqrt{2}}{16}$ & $2.886\times 10^{-1}$ & -- & $2.907\times 10^{-1}$ & -- & $2.943\times 10^{-1}$ & --
	\\
$\nicefrac{\sqrt{2}}{32}$ & $1.455\times 10^{-1}$ & 0.99 & $1.462\times 10^{-1}$ & 0.99 & $1.466\times 10^{-1}$ & 1.01
	\\
$\nicefrac{\sqrt{2}}{64}$ & $7.290\times 10^{-2}$ & 1.00 & $7.320\times 10^{-2}$ & 1.00 & $7.313\times 10^{-2}$ & 1.00
	\\
$\nicefrac{\sqrt{2}}{128}$ & $3.647\times 10^{-2}$ & 1.00 & $3.660\times 10^{-2}$ & 1.00 & $3.653\times 10^{-2}$ & 1.00
	\\
$\nicefrac{\sqrt{2}}{256}$ & $1.824\times 10^{-2}$ & 1.00 & $1.839\times 10^{-2}$ & 1.00 & $1.826\times 10^{-2}$ &  1.00
	\\
	\hline
	\end{tabular}
\caption{$H^1$ convergence test. The final time is $T = 1.0$, and the 
refinement path is taken to be $\tau = 1.6h$. The other parameters are 
$\varepsilon = \gamma = 1.0$; $\Omega = (0,1)^2$. The global error at $T$ 
is expected to be $\mathcal{O}(\tau)+\mathcal{O}(h) = \mathcal{O}(h)$, and 
this is confirmed.}
	\label{tab2}
	\end{table}

For the convergence of the problem with source terms, we solve a
problem of the following form: find $\left(p^m_h,
\mu^m_h,\vphi^m_h\right)\in \mathring{V}_h\times V_h\times V_h$, such that
        \begin{alignat}{2}
\bigl(\nab\phm+\gamma \vphi_h^{m-1}\nab\mu_h^m,\nab q_h\bigr) 
&= \prodt{s_1({\bf x},t_m)}{q_h}  &&\quad \forall \  q_h\in V_h, 
\label{with-source-1} \\
\prodt{d_t\vphihm}{\nu_h} +\vep\bigl(\nab \mu_h^m,\nab\nu_h \bigr) 
\hskip 1.0in & && \label{with-source-2} \\
+\prodt{\vphi_h^{m-1}[\nab \phm+\gamma\vphi_h^{m-1}\nab \mu_h^m]}{\nab \nu_h} 
&= \prodt{s_2({\bf x},t_m)}{\nu_h}  &&\quad\forall \ \nu_h\in V_h, \nonumber\\
\prodt{\mu_h^m}{\psi_h}-\vep \prodt{\nab\vphihm}{\nab\psi_h}
-\frac{1}{\vep} \prodt{f_h^m}{\psi_h} &=\prodt{s_3({\bf x},t_m)}{\psi_h} 
&&\quad\forall \ \psi_h\in V_h, \label{with-source-3} \\
\vphi_h^0 &=\vphi_{0h}, && \label{with-source-4}
        \end{alignat}
for $m=1,\ldots, M$, where the source terms are chosen so that the
solution of the corresponding continuous problem is precisely
        \begin{equation}
p(x,y,t)= \mu(x,y,t) = \vphi(x,y,t) = \cos(\pi t) \cdot g(x)\cdot g(y)  ,
        \end{equation}
with $g(\xi) = 16\xi^2(\xi-1)^2$.  The initial data are precisely given
by $\vphi_{0h} = \mathcal{I}_h\left(\vphi(\, \cdot\, ,0)\right)$,
where $\mathcal{I}_h :   H^2\left(\Omega\right) 
\to V_h$ is the standard nodal interpolation operator.  All integrations are done exactly using the appropriate Gauss-quadrature rules.  This is of course made possible since we are using polynomials in space.  The exact values of all of the other parameters used in the test are given in the captions of Tabs.~\ref{tab1} and \ref{tab2}. The
results of an $L^2$ error analysis using a quadratic refinement path
are found in Tab.~\ref{tab1} and confirm the expected optimal second-order convergence rate in this
case. The results of an $H^1$ error analysis using a linear refinement
path are found in Tab.~\ref{tab2} and confirm the expected optimal first-order convergence rate for
this case.  Notice that the approximations $p^m_h$, $\vphi^m_h$, and $\mu^m_h$ all appear to converge at the same optimal rates, in both cases.

	\begin{table}[t]
	\centering
	\begin{tabular}{ccccccccc}
$h_c$ & $h_f$ & $\norm{\delta_\vphi}{L^2}$ & rate & $\norm{\delta_\mu}{L^2}$ & rate & $\norm{\delta_p}{L^2}$ & rate
	\\
	\hline
$\nicefrac{\sqrt{2}}{16}$ & $\nicefrac{\sqrt{2}}{32}$ & $5.514\times 10^{-2}$ & -- & $2.890\times 10^{-1}$ & -- & $3.099\times 10^{-2}$ & --
	\\
$\nicefrac{\sqrt{2}}{32}$ & $\nicefrac{\sqrt{2}}{64}$ & $2.165\times 10^{-2}$ & 1.35 & $1.229\times 10^{-1}$ & 1.23 & $1.148\times 10^{-2}$ & 1.43
	\\
$\nicefrac{\sqrt{2}}{64}$ & $\nicefrac{\sqrt{2}}{128}$ & $6.284\times 10^{-3}$ & 1.78 & $3.588\times 10^{-2}$ & 1.78 & $3.250\times 10^{-3}$ & 1.82
	\\
$\nicefrac{\sqrt{2}}{128}$ & $\nicefrac{\sqrt{2}}{256}$ & $1.636\times 10^{-3}$ & 1.94 & $9.327\times 10^{-3}$ & 1.94 & $8.420\times 10^{-4}$ & 1.95
	\\Ä
$\nicefrac{\sqrt{2}}{256}$ & $\nicefrac{\sqrt{2}}{512}$ & $4.132\times 10^{-4}$ & 1.99 & $2.355\times 10^{-3}$ & 1.99 & $2.128\times 10^{-4}$ &  1.98
	\\
	\hline
	\end{tabular}
\caption{$L^2$ Cauchy convergence test. The final time is 
$T = 4.0\times 10^{-2}$, and the refinement path is taken to be 
$\tau = 1.024h^2$. The other parameters are $\varepsilon =6.25\times 10^{-2}$;
$\gamma = 1.25\times 10^{-1}$; $\Omega = (0,1)^2$.  The Cauchy difference is 
defined via $\delta_\vphi := \vphi_{h_f}-\vphi_{h_c}$, where the 
approximations are evaluated at time $t=T$, and analogously for 
$\delta_\mu$, and $\delta_p$.  The norm of the Cauchy difference 
at $T$ is expected to be $\mathcal{O}(\tau)+\mathcal{O}\left(h^2\right) 
= \mathcal{O}\left(h^2\right)$.}
	\label{tab3}
	\end{table}

	\begin{table}
	\centering
	\begin{tabular}{ccccccccc}
$h_c$ & $h_f$ & $\norm{\delta_\vphi}{H^1}$ & rate & $\norm{\delta_\mu}{H^1}$ & rate & $\norm{\delta_p}{H^1}$ & rate
	\\
	\hline
$\nicefrac{\sqrt{2}}{16}$ & $\nicefrac{\sqrt{2}}{32}$ & $8.569\times 10^{-1}$ & -- & $1.301\times 10^{-0}$ & -- & $8.371\times 10^{-2}$ & --
	\\
$\nicefrac{\sqrt{2}}{32}$ & $\nicefrac{\sqrt{2}}{64}$ & $4.160\times 10^{-1}$ & 1.04 & $6.295\times 10^{-1}$ & 1.04 & $3.715\times 10^{-1}$ & 1.17
	\\
$\nicefrac{\sqrt{2}}{64}$ & $\nicefrac{\sqrt{2}}{128}$ & $2.061\times 10^{-1}$ & 1.01& $3.111\times 10^{-1}$ & 1.02 & $1.779\times 10^{-2}$ & 1.06
	\\
$\nicefrac{\sqrt{2}}{128}$ & $\nicefrac{\sqrt{2}}{256}$ & $1.029\times 10^{-1}$ & 1.00& $1.554\times 10^{-1}$ & 1.00 & $8.834\times 10^{-3}$ & 1.01
	\\
$\nicefrac{\sqrt{2}}{256}$ & $\nicefrac{\sqrt{2}}{512}$ & $5.146\times 10^{-2}$ & 1.00 & $7.777\times 10^{-2}$ & 1.00 & $4.422\times 10^{-3}$ &  1.00
	\\
	\hline
	\end{tabular}
\caption{$H^1$ Cauchy convergence test. The final time is 
$T = 4.0\times 10^{-2}$, and the refinement path is taken to 
be $\tau = 2.0\times 10^{-3}h$. The other parameters are 
$\varepsilon =6.25\times 10^{-2}$;  $\gamma = 1.25\times 10^{-1}$; 
$\Omega = (0,1)^2$.  The norm of the Cauchy difference at $T$ is expected to 
be $\mathcal{O}(\tau)+\mathcal{O}\left(h\right) = \mathcal{O}\left(h\right)$.}
	\label{tab4}
	\end{table}

We now give the results of a test without any artificial sources. In other 
words, we solve the scheme (\ref{with-source-1})--(\ref{with-source-3}) 
with $s_i \equiv 0$, $i=1,2,3$.  The initial data are taken to be
	\begin{equation}
\vphi_{0h} = \mathcal{I}_h\left( \frac{\big[1.0-\cos(4.0\pi x)\big]\cdot \big[1.0-\cos(2.0\pi y)\big]}{2}-1.0\right) ,
	\end{equation}
and the parameters are given in the captions of Tabs.~\ref{tab3} and \ref{tab4}. Note that in this case we are not in possession of the exact solutions.   To circumvent this, we measure the difference of the computed
solutions at successive resolutions.  Specifically, we compute the rate 
at which the Cauchy difference $\delta_\psi := \psi^{M_f}_{h_f} 
- \psi^{M_c}_{h_c}$ converges to zero, where $h_f=2h_c$, $\tau_f 
= 2^p\tau_c$ ($p=1$ for a linear refinement path and $p=2$ for a 
quadratic refinement path), and $\tau_fM_f = \tau_cM_c=T$. A quadratic 
refinement path, \emph{i.e.}, $\tau = Ch^2$, is taken when measurements 
are made in the $L^2$ norm,  and a linear refinement path, \emph{i.e.}, 
$\tau = Ch$, when measurements are made in the $H^1$ norm. The results of an $L^2$ Cauchy error analysis are found 
in Tab.~\ref{tab3} and confirm second-order convergence in this case.  The 
results of an $H^1$ Cauchy error analysis are found in Tab.~\ref{tab4} and 
confirm first-order convergence in this case.

Our final test is a simulation of spinodal decomposition with different values of $\gamma$.  Specifically, we solve the scheme (\ref{with-source-1})--(\ref{with-source-3}) with $s_i \equiv 0$, $i=1,2,3$, and with three values of $\gamma$; namely, $\gamma = 0$, which yields
the familiar Cahn-Hilliard model; $\gamma = 0.01$; and $\gamma = 0.04$.  
Furthermore, we take 
$\vep = 0.01$, $h = \frac{\sqrt{2}}{256}$, $\tau = 1\times 10^{-3}$, 
and $T = 0.1$.  We use the same randomized initial data for the three 
simulations represented in Fig.~\ref{fig1}, where the average value 
of $\vphi$ is approximately $-0.1$.  As expected, the mixture phase 
separates into domains wherein $\vphi\approx -1$ and 
$\vphi\approx +1$. Afterwards the system coarsens, as larger phase regions
grow at the expense of smaller ones.  The energy for the three 
simulations is displayed in Fig.~\ref{fig2}.  A general trend emerges, 
where, at least at early times, the energy decreases faster and the 
coarsening process is appears to be accelerated as the excess surface
tension $\gamma$ increases.  This behavior is expected and was observed
in similar finite difference calculations undertaken in~\cite{wise10}.

	\begin{figure}
	\begin{center}
	\includegraphics[width=5in]{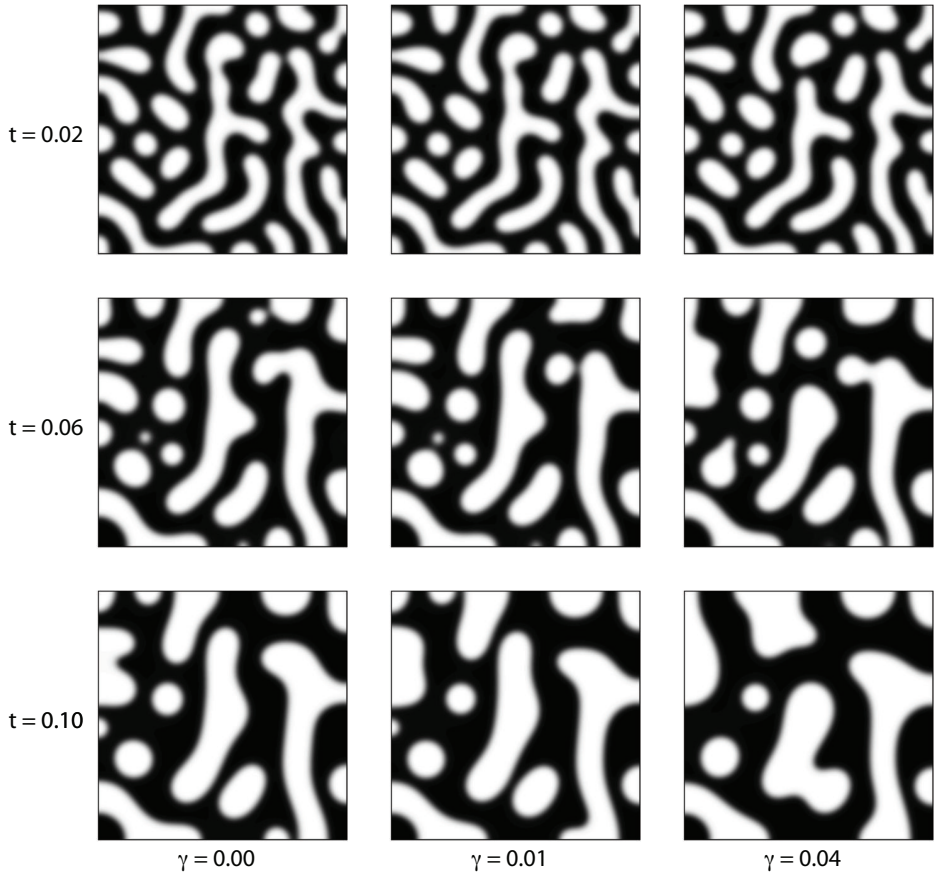}
\caption{Spinodal decomposition for three values of $\gamma$.  The domain 
is $\Omega = (0,1)\times(0,1)$ and $\vep = 0.01$.  The initial data 
are exactly the same for the three simulations.  The time step size is 
$\tau = 1.0\times 10^{-3}$, and $h = \nicefrac{\sqrt{2}}{256}$.  We use a 
uniform mesh, as in Fig.~\ref{fig3}.  The corresponding energy plots are 
shown in Fig.~\ref{fig2}.  The average value of $\vphi$ for all three 
simulations is approximately $-0.1$.  For $\gamma=0.01$, the mass variation 
over the simulated time is only $1\times 10^{-12}$.  The max and min values 
of $\vphi$ are very near the values $+1$ and $-1$, respectively.}
	\label{fig1}
	\end{center}
	\end{figure}

	\begin{figure}
	\begin{center}
	\includegraphics[width=5in]{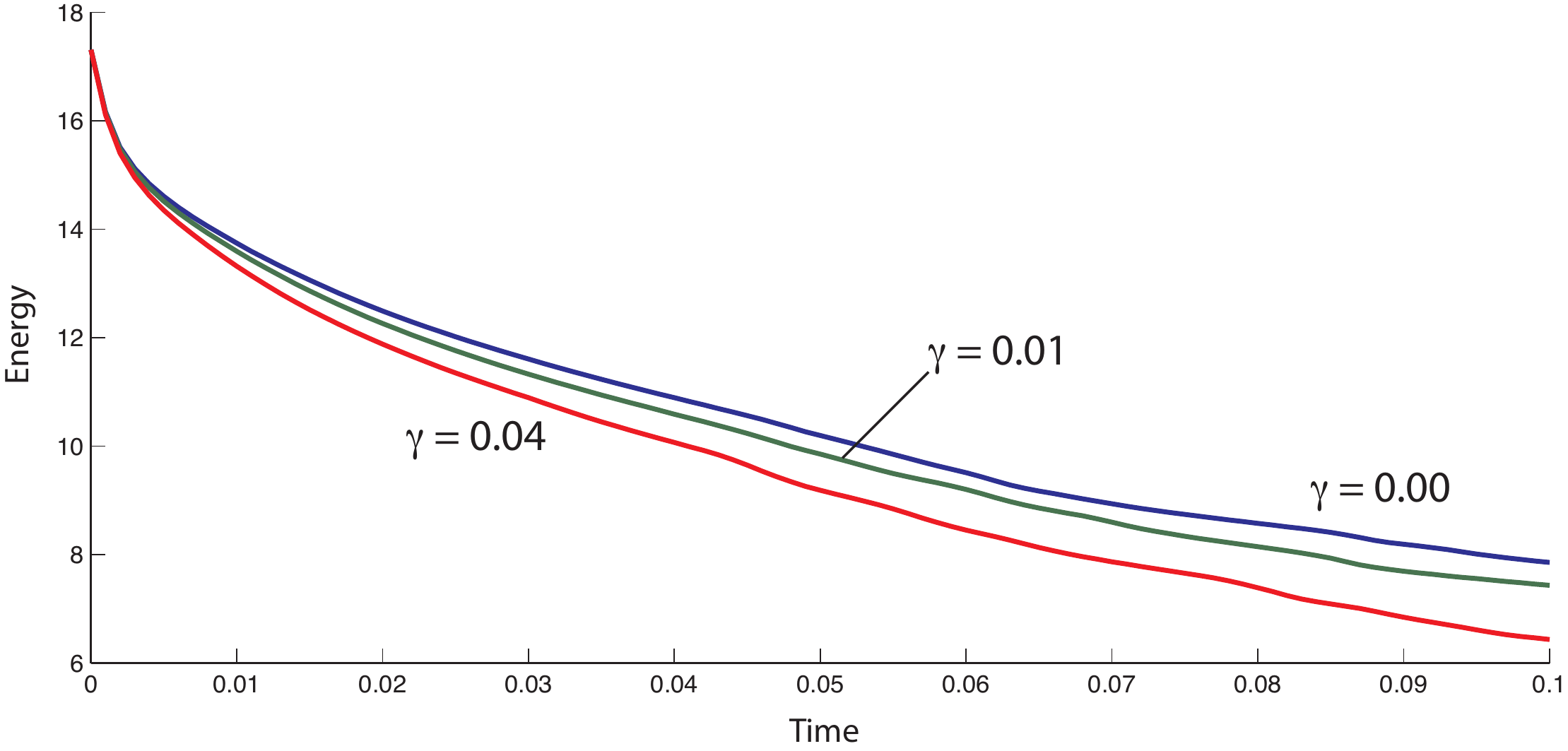}
\caption{Energy plots for the spinodal decomposition simulations depicted 
in Fig.~\ref{fig1}. The parameters for the simulations are given in the 
caption of Fig.~\ref{fig1}.  The energy is observed to decrease at each 
time step.  The general trend, at least at early times, is that the energy 
decreases faster with increasing values of the excess surface tension $\gamma$.}
	\label{fig2}
	\end{center}
	\end{figure}

Note that we have proved that (at the theoretical level) the energy is 
non-increasing at each time step.  This is observed in our computations.  
In addition to this, the mass, \emph{i.e},  $\int_\Omega \vphi_h \, dx$, 
at the theoretical level is expected to be unchanging from one time step 
to the next.  On the practical level, we observe very little mass variation.  
For example, for the $\gamma=0.01$ case depicted in Fig.~\ref{fig1}, where 
initially $\int_\Omega \vphi_{0h} \, dx \approx -0.1$,  we observe mass 
variation of only $1\times 10^{-12}$ over the whole of the simulation time.  
Note that our multigrid iteration stopping tolerance is of the same order, 
namely,  $tol = 1\times 10^{-12}$ in (\ref{stopping-tolerance}).

	
\section*{Acknowledgments} 
The authors would like to thank Professor Xiaoming Wang of Florida State
University for his helpful discussions and for bringing the reference
\cite{Wang_Zhang10} to their attention.

\appendix
\section{Nonlinear Multigrid Solver} \label{app-multigrid-solver}
In this appendix we give the full details of the nonlinear multigrid solver 
that is used to march the scheme in time.  Suppose $\Omega\subseteq \mathbb{R}^2$ 
is polygonal, and assume that $\Tc_\ell$, $\ell = 0,1,\ldots,L$, is a 
hierarchy of nested 
triangulations of $\Omega$ as suggested in Fig.~\ref{fig1}.  In particular, 
$\Tc_{\ell}$ is obtained by subdividing the triangles of $\Tc_{\ell-1}$ 
into 4 congruent sub-triangles.  Note that $h_{\ell-1} = 2h_\ell$, 
$\ell=1,\ldots, L$, and that $\left\{\Tc_{\ell}\right\}$ is a quasi-uniform 
family.  For simplicity, we shall use $P_1$ finite element spaces and use the 
same space for the pressure as is used for the other variables.  We define
\begin{displaymath}
V_\ell=\left\{v\in C^0(\overline{\Ome}) \ \middle| \ v|_K\in P_1(K)
\ \ \forall K\in \Tc_\ell \right\}, 
\end{displaymath}
for $\ell = 0, \ldots, L$ and observe the nested space chain 
$V_0 \subset V_{1} \subset V_{3} \subset \cdots \subset V_{L}$. Because of 
this nestedness, there is a natural injection operation 
$I_{\ell-1,\ell} : V_{\ell-1} \hookrightarrow V_\ell$ defined 
by $I_{\ell-1,\ell}(v) = v$, for all $v\in V_{\ell-1}$,  
$\ell = 1, \ldots , L$. Now, let $B_\ell = \left\{u_{\ell,i}(\bx)\right\}_{i=1}^{N_\ell}$ be the nodal basis 
for $V_\ell$, $\ell = 0, 1, \ldots , L$.  In other words, 
$u_{\ell,j}\left({\bf x}_{\ell,i}\right)=\delta_{i,j}$, where 
$\left\{{\bf x}_{\ell,i}\right\}_{i=1}^{N_\ell}$ are the nodes of 
$\Tc_\ell$.  We have level-wise representations of the 
unknowns of the form
\begin{equation}
\varphi_\ell(\bx) = \sum_{i=1}^{N_\ell} \varphi_{\ell,i} u_{\ell,i}(\bx) 
\iff \bfvphi_\ell = \left(\vphi_{\ell,1},\vphi_{\ell,2},\ldots ,
\vphi_{\ell,N_\ell}\right)^T  , 
\end{equation}
and similarly for $\mu_\ell(\bx)$ and $p_\ell(\bx)$. Define the prolongation 
matrix via $\Psf_{\ell-1,\ell} := \Isf_{\ell-1,\ell}$, 
where $\Isf_{\ell-1,\ell}$ is the $N_\ell\times N_{\ell-1}$ matrix 
representation of the injection operator $I_{\ell-1,\ell}$ with respect 
to the bases $B_{\ell-1}$ and $B_\ell$.  There are two restriction 
operations --- \emph{i.e.}, operations transferring information from the 
finer space $V_\ell$ to the coarser space $V_{\ell-1}$ --- that we shall use.  
The first is called the canonical 
restriction and, in matrix form, is the $N_{\ell-1}\times N_\ell$ 
matrix defined via $\Rsf_{\ell,\ell-1} := \Isf_{\ell-1,\ell}^T$
\cite{braess07, bs08}.  The second is defined via
\begin{equation}
\hat{R}_{\ell,\ell-1}(v) = \sum_{i = 1}^{N_{\ell-1}} v
\left({\bf x}_{\ell-1,i}\right)u_{\ell-1,i}\left({\bf x}\right) 
\quad \forall v\in V_\ell, 
\end{equation}
where the points ${\bf x}_{\ell-1,i}$ are the nodes of the mesh 
$\Tc_{\ell-1}$.  Note that 
$\left\{{\bf x}_{\ell-1,i}\right\}_{i=1}^{N_{\ell-1}}\subset 
\left\{{\bf x}_{\ell,i}\right\}_{i=1}^{N_\ell}$ by construction. 
By $\hat{\Rsf}_{\ell,\ell-1}$ we denote the matrix representation of 
$\hat{R}_{\ell,\ell-1}$  with respect to the bases 
$B_\ell$ and $B_{\ell-1}$.

In the present framework, our nonlinear finite element scheme is defined 
on the finest level, $\ell = L$, as follows: find the triple 
$(p_L, \mu_L, \vphi_L) \in \mathring{V}_L\times V_L\times V_L$ such that  
\begin{alignat}{2}
\bigl(\nab p_L+\gamma \vphi_L^{m-1}\nab\mu_L,\nab q_L\bigr) &=0 
&&\quad \forall q_L\in \mathring{V}_L \ , \label{eqA3.2} \\
\prodt{\vphi_L}{\nu_L} +\tau\vep\bigl(\nab \mu_L,\nab\nu_L \bigr) 
\hskip 1.0in & && \label{eqA3.3} \\
+\tau\prodt{\vphi_L^{m-1}\left[\nab p_L 
+\gamma\vphi_L^{m-1}\nab \mu_L\right]}{\nab \nu_L}  
&= \prodt{\varphi_L^{m-1}}{\nu_L} &&\quad\forall \nu_L\in V_L, \nonumber \\
\prodt{\mu_L}{\psi_L}-\vep \prodt{\nab\vphi_L}{\nab\psi_L} 
-\frac{1}{\vep} \prodt{\left(\varphi_L\right)^2\varphi_L}{\psi_L} 
&=-\frac{1}{\vep} \prodt{\varphi_L^{m-1}}{\psi_L} 
&&\quad\forall \psi_L\in V_L, \label{eqA3.4}
\end{alignat}
where $\varphi_L^{m-1}\in V_L$ is given.  We have dropped the superscript 
$m$ (the time step index) on the unknowns for simplicity. Theorem \ref{thm-existence-uniqueness} guarantees that this problem always has a unique 
solution. The nonlinear system (\ref{eqA3.2})--(\ref{eqA3.4}) may be written as 
	\begin{eqnarray}
\Asf_L\bp_L+\gamma \Csf_L\bfmu_L &=& {\bf 0} , 
	\label{matrix-sys-1} 
	\\
\Msf_L\bfvphi_L + \tau\left(\vep\Asf_L+\gamma \Bsf_L\right) \bfmu_L  + \tau\Csf_L\bp_L &=& \Msf_L\bfvphi_L^{m-1}  ,
	\label{matrix-sys-2} 
	\\
\vep\Asf_L\bfvphi_L+\frac{1}{\vep}\Qsf_L\left(\bfvphi_L\right)\bfvphi_L - \Msf_L\bfmu_L &=& \frac{1}{\vep} \Msf_L\bfvphi_L^{m-1}, 
	\label{matrix-sys-3}
	\end{eqnarray}
where $\Asf_L$, $\Bsf_L$, $\Csf_L$, $\Msf_L$, and $\Qsf_L\left(\bfvphi_L\right)$ are $N_L\times N_L$ matrices whose components are
	\begin{alignat}{2}
\left[\Asf_L\right]_{i,j} &:= \prodt{\n u_{L,j}}{\n u_{L,i}} , &
\quad \left[\Bsf_L\right]_{i,j} &:= \prodt{\left(\vphi_L^{m-1}\right)^2 
\n u_{L,j}}{\n u_{L,i}}  , 
	\\
\left[\Csf_L\right]_{i,j} &:= \prodt{\vphi_L^{m-1} \n u_{L,j}}{\n u_{L,i}}, &  \quad \left[\Msf_L\right]_{i,j} &:= \prodt{u_{L,j}}{u_{L,i}}, 
	\\
\left[\Qsf_L\left(\bfvphi_L\right)\right]_{i,j}  &:= \prodt{\left(\vphi_L\right)^2 u_{L,j}}{u_{L,i}} . & &
	\end{alignat}

\begin{figure}
\begin{center}
\includegraphics[width=5.1in]{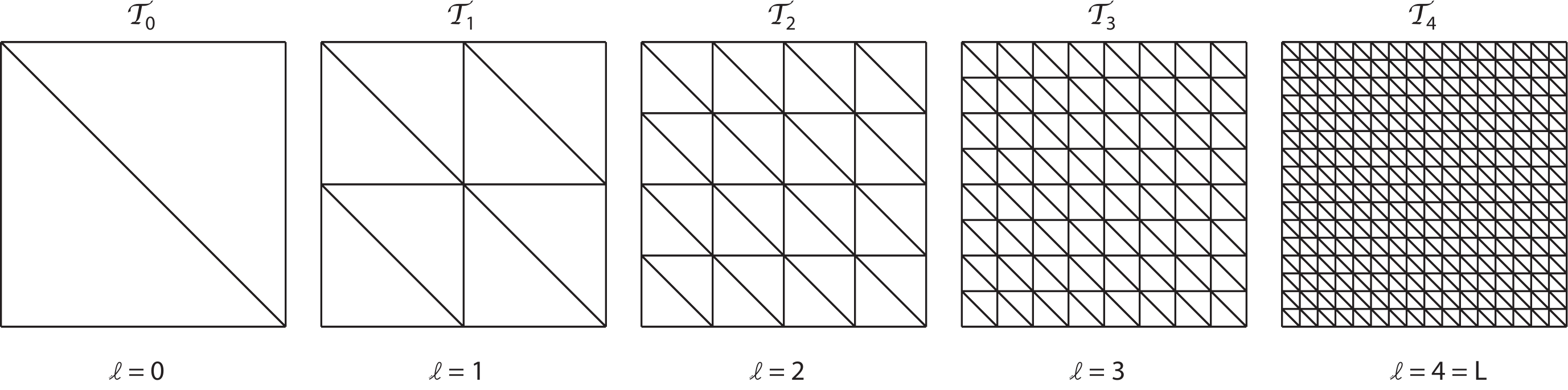}
\caption{A hierarchical triangulation, $\Tc_\ell$, $\ell = 0,1,\ldots,L$, of a square domain $\Omega$.  Here $L = 4$, though in typical calculations we may use $L = 8$ or 9. } \label{fig3}
\end{center}
\end{figure}

We solve (\ref{matrix-sys-1})--(\ref{matrix-sys-3}) using a nonlinear multigrid method~\cite[Ch.~5, \S 6]{braess07}.  This requires that we split the equations into 
source (${\bf s}$) and operator ($\Nsf$) terms:
	\begin{alignat}{2}
{\bf s}_L^{(1)} &:= {\bf 0}, & \quad \Nsf_L^{(1)}\left(\bfphi_L\right) &:= \Asf_L\bp_L+\gamma \Csf_L\bfmu_L \ , 
	\label{nonlinear-operator-1}
	\\
{\bf s}_L^{(2)} &:= \Msf_L\bfvphi_L^{m-1}  , & \quad \Nsf_L^{(2)}\left(\bfphi_L\right) &:= \Msf_L\bfvphi_L  + \tau\left(\vep\Asf_L+\gamma \Bsf_L\right) \bfmu_L + \tau\Csf_L\bp_L ,
	\label{nonlinear-operator-2}
	\\
{\bf s}_L^{(3)} &:= \frac{1}{\vep} \Msf_L\bfvphi_L^{m-1}, & \quad \Nsf_L^{(3)}\left(\bfphi_L\right) &:= \vep\Asf_L\bfvphi_L
+\frac{1}{\vep}\Qsf_L\left(\bfvphi_L\right)\bfvphi_L - \Msf_L\bfmu_L  ,
	\label{nonlinear-operator-3}
	\end{alignat}
where $\bfphi_L := \left[\bp_L,\bfmu_L,\bfvphi_L \right]$ is the 
$N_L\times 3$ array of unknowns.  We must also define a ``consistent" version of the nonlinear operator on all of the coarser levels. There are a number of ways to proceed in this task~\cite{braess07}; we choose the following path.   Suppose that $\ell\in\left\{0,1,\ldots, L-1\right\}$ is given. We restrict the known solution from the previous time step to the coarser levels via
	\begin{eqnarray}
\bfvphi_\ell^{m-1}:= \prod_{j = \ell+1}^L\hat{\Rsf}_{j,j-1}\bfvphi_{L}^{m-1}
&=&\left(\vphi_{\ell,1}^{m-1},\vphi_{\ell,2}^{m-1},\ldots 
\vphi_{\ell,N_\ell}^{m-1}\right)^T, 
	\\
\iff \varphi_\ell^{m-1}(\bx) &=&\sum_{i=1}^{N_\ell} \varphi_{\ell,i}^{m-1} u_{\ell,i}(\bx) .
	\nonumber
	\end{eqnarray}
Now, given any $\psi_\ell\in V_\ell$ with the representation
	\begin{equation}
\bfpsi_\ell = \left(\psi_{\ell,1},\psi_{\ell,2},\ldots,\psi_{\ell,N_\ell}  \right)^T \iff \psi_\ell(\bx)=\sum_{i=1}^{N_\ell}\psi_{\ell,i}u_{\ell,i}(\bx),
	\end{equation}
we define
	\begin{alignat}{2}
\left[\Asf_\ell\right]_{i,j} &:= \prodt{\n u_{\ell,j}}{\n u_{\ell,i}}, &  \quad \left[\Bsf_\ell\right]_{i,j}  &:= \prodt{\left(\vphi_\ell^{m-1}\right)^2 \n u_{\ell,j}}{\n u_{\ell,i}}, 
	\\
\left[\Csf_\ell\right]_{i,j} &:= \prodt{\vphi_\ell^{m-1} \n u_{\ell,j}}{\n u_{\ell,i}}, & \quad \left[\Msf_\ell\right]_{i,j} &:= \prodt{u_{\ell,j}}{u_{\ell,i}}, 
	\\
\left[\Qsf_\ell\left(\bfpsi_\ell\right)\right]_{i,j} &:= \prodt{\left(\psi_\ell\right)^2 u_{\ell,j}}{u_{\ell,i}}. & &
 	\end{alignat}
Observe that 
\begin{equation}
\Asf_\ell = \Rsf_{\ell+1,\ell}\Asf_{\ell+1}\Psf_{\ell,\ell+1}, 
\quad \Msf_\ell = \Rsf_{\ell+1,\ell}\Msf_{\ell+1}\Psf_{\ell,\ell+1},
\end{equation}
which is standard in the finite element setting~\cite{braess07, bs08} 
and is the reason for the term ``canonical" describing $\Rsf_{\ell+1,\ell}$. 
On the other hand,
\begin{equation}
\Bsf_\ell \approx \Rsf_{\ell+1,\ell}\Bsf_{\ell+1}\Psf_{\ell,\ell+1}, 
\quad \Csf_\ell \approx \Rsf_{\ell+1,\ell}\Csf_{\ell+1}\Psf_{\ell,\ell+1}.
\end{equation}
(Note that we could have recursively defined 
$\Bsf_\ell = \Rsf_{\ell+1,\ell}\Bsf_{\ell+1}\Psf_{\ell,\ell+1}$, and similarly for $\Csf_\ell$.  
But it turns out that this is an unnecessary complication from the point of view of the convergence of the algorithm.) Finally, we have 
	\begin{eqnarray}
\Nsf_\ell^{(1)}\left(\bfxi_\ell\right) &:=& \Asf_\ell\bq_\ell +\gamma \Csf_\ell\bfnu_\ell,
	\label{nonlinear-operator-k-1}
	\\
\Nsf_\ell^{(2)}\left(\bfxi_\ell\right) &:=& \Msf_\ell\bfpsi_\ell  + \tau\left(\vep\Asf_\ell+\gamma \Bsf_\ell\right) \bfnu_\ell  + \tau\Csf_\ell\bq_\ell,
	\label{nonlinear-operator-k-2}
	\\
\Nsf_\ell^{(3)}\left(\bfxi_\ell\right) &:=& \vep\Asf_\ell\bfvphi_\ell +\frac{1}{\vep}\Qsf_\ell\left(\bfpsi_\ell\right)\bfpsi_\ell  - \Msf_\ell\bfnu_\ell, 
	\label{nonlinear-operator-k-3}
	\end{eqnarray}
where $\bfxi_\ell := \left[\bq_\ell,\bfnu_\ell,\bfpsi_\ell \right]$ is 
any given $N_\ell\times 3$ array of unknowns.

We are now in a position to define the recursive nonlinear multigrid V-Cycle operator~\cite[Ch.~5, \S 6]{braess07}, 
which is the heart of our solver. In the following the 
superscript $k$ is the V-Cycle loop index (not the time step index).  Let 
$\bfphi_\ell^{k-1} := \left[\bp_\ell^{k-1},\bfmu_\ell^{k-1},\bfvphi_\ell^{k-1} \right]$ 
denote the current, level-$\ell$ multigrid iterate.   
For any $N_\ell\times 3$ array of unknowns $\bfxi_\ell$, define $\Nsf_\ell\left(\bfxi_\ell\right) 
:= \left[\Nsf_\ell^{(1)}\left(\bfxi_\ell\right), 
\Nsf_\ell^{(2)}\left(\bfxi_\ell\right), 
\Nsf_\ell^{(3)}\left(\bfxi_\ell\right)\right]$, and 
${\bf s}_\ell := \left[{\bf s}_\ell^{(1)}, {\bf s}_\ell^{(2)}, 
{\bf s}_\ell^{(3)}\right]$. Note that these 
last two objects are $N_\ell \times 3$ arrays by design. We define the action of the recursive  nonlinear multigrid V-Cycle operator
\begin{equation}
\bfphi_\ell^k = \mbox{NMGM}
\left(\ell,\bfphi_\ell^{k-1},\Nsf_\ell,{\bf s}_\ell,\lambda\right)
\end{equation}
in the following 3 steps:
\begin{enumerate}
\item
Pre-smoothing:
\begin{itemize}
\item
Given $\bfphi_\ell^{k-1}$, compute a smoothed level-$\ell$ approximation $\bar\bfphi_\ell$:
\begin{equation}
\bar\bfphi_\ell = {\mathcal S}^\lambda
\left(\bfphi_\ell^{k-1},\Nsf_\ell,{\bf s}_\ell \right) ,
\end{equation}
where ${\mathcal S}$ is a smoothing (or relaxation) operator, and $\lambda >0$ is the number of smoothing sweeps. 
\end{itemize}
	
\item
Coarse-grid correction: 
\begin{itemize}
\item
Compute coarse-level initial iterate:
\begin{equation}
\bar\bfphi_{\ell-1} = \hat{\Rsf}_{\ell,\ell-1} \bar\bfphi_\ell \ .
\end{equation}
\item
Compute the coarse-level right-hand side:
\begin{equation}
{\bf s}_{\ell-1}  = \Rsf_{\ell,\ell-1} \left( {\bf s}_\ell
-\Nsf_\ell( \bar\bfphi_\ell) \right) 
+ \Nsf_{\ell-1}\left( \bar\bfphi_{\ell-1} \right).
\end{equation}

\item
Compute an approximate solution $\hat{\bfxi}_{\ell-1}$ of the following 
coarse grid  equation:
\begin{equation}
\Nsf_{\ell-1}(\bfxi_{\ell-1}) = {\bf s}_{\ell-1}.  \label{coeq}
\end{equation}
Note that this equation is uniquely solvable by 
Theorem \ref{thm-existence-uniqueness}.
	
\begin{itemize}
\item If $\ell=1$ employ $\lambda$ smoothing steps:
\begin{equation}
\hat\bfxi_0 = \mathcal{S}^\lambda 
\left(\bar\bfphi_0, \Nsf_0,{\bf s}_0\right).
\end{equation}

\item
If $\ell>1$ get an approximate solution to Eq. (\ref{coeq}) 
using $\bar\bfphi_{\ell-1}$ as initial guess:
\begin{equation}
\hat{\bfxi}_{\ell-1} = \mbox{NMGM}\left(\ell-1,\bar\bfphi_{\ell-1},
\Nsf_{\ell-1},{\bf s}_{\ell-1},\lambda\right).
\end{equation}
\end{itemize}

\item
Compute the coarse-grid correction:
\begin{equation}
\hat{\bfphi}_{\ell-1} = \hat{\bfxi}_{\ell-1} - \bar{\bfphi}_{\ell-1}.
\end{equation}

\item
Compute the coarse-grid-corrected approximation at level $k$:
\begin{equation}
\hat{\bfphi}_\ell=\Psf_{\ell-1,\ell}\hat{\bfphi}_{\ell-1}+\bar{\bfphi}_\ell.
\end{equation}
	
\end{itemize}
	
\item
Post-smoothing:
\begin{itemize}
\item
Finally, compute $\bfphi_\ell^k$ by applying $\lambda$ smoothing steps:
\begin{equation}
\bfphi_\ell^k = \mathcal{S}^\lambda 
\left(\hat{\bfphi}_\ell,\Nsf_\ell, {\bf s}_\ell\right).
\end{equation}
	
\end{itemize}
\end{enumerate}
When
	\begin{equation}
\sqrt{\frac{1}{3N_L}\sum_{j=1}^3\sum_{i=1}^{N_L}\left(\left[{\bf s}^{(j)}_L-\Nsf^{(j)}_L\left(\bfphi_L^k\right)\right]_i\right)^2 \ } < tol
	\label{stopping-tolerance}
	\end{equation}
we stop iterating and set 
$\bfphi_L^k \to \bfphi_L=\left[\bp_L,\bfmu_L,\bfvphi_L \right]$, the 
fine-level solution.  For smoothing, we use a nonlinear block Gau\ss-Seidel 
method, like that discussed in~\cite{wise10} for a similar finite-difference 
nonlinear multigrid method. The exact details are omitted for brevity, but the principal idea is that the nodal values $\left(p_{\ell,i},\mu_{\ell,i},\vphi_{\ell,i} \right)$ are always obtained simultaneously in the smoothing operation.   We use $\lambda=2$ or 3 in the smoothing step.


\end{document}